\documentclass[12pt]{amsart}

\usepackage{fullpage}
\usepackage{amssymb,amsthm,amsmath,amsrefs,amsfonts, amscd}
\usepackage{dsfont}
\usepackage[all]{xy}

\numberwithin{equation}{section}

\newtheorem{theorem}{Theorem}
\newtheorem{corollary}{Corollary}
\newtheorem{lemma}{Lemma}
\newtheorem{proposition}{Proposition}
\newtheorem{definition}{Definition}

\newcommand{\K}{\mathcal{K}}
\newcommand{\N}{\mathbb{N}}
\newcommand{\Z}{\mathbb{Z}}
\newcommand{\R}{\mathbb{R}}
\newcommand{\C}{\mathbb{C}}
\newcommand{\Cu}{\mathrm{Cu}}
\newcommand{\CCu}{\mathbf{Cu}}
\newcommand{\Hom}{\mathrm{Hom}}
\newcommand{\Mor}{\mathrm{Mor}}
\newcommand{\id}{\mathrm{id}}
\newcommand{\E}{\mathrm{E}}
\newcommand{\Ad}{\mathrm{Ad}}

\newcommand{\M}{\mathrm{M}}

\title{On inductive limits of type I C*-algebras with one-dimensional spectrum}

\author{Alin Ciuperca, George A. Elliott, \and Luis Santiago}

\address{Alin Ciuperca, Department of Mathematics and Statistics\\
         University of New Brunswick\\
     Fredericton,  New Brunswick E3B 5A3, Canada}
\email{alinc@unb.ca}    
\address{George A. Elliott, Department of Mathematics\\
         University of Toronto\\
         Toronto, Ontario M5S 2E4, Canada}
\email{elliott@math.toronto.edu} 

\address{Luis Santiago, Departament de Matem\`atiques \\Edifici C,
         Universitat Aut\`onoma de Barcelona\\ Bellaterra, Barcelona 08193, Spain}
\email{santiago@mat.uab.cat}

\begin{document}
\begin{abstract}
The class of separable C*-algebras which can be written as inductive limits of continuous-trace C*-algebras with spectrum homeomorphic to a disjoint union of trees and trees with a point removed is classified by the Cuntz semigroup.
\end{abstract}

\maketitle

\section{Introduction}
In recent years the Cuntz semigroup has become an important tool of investigation in the theory of C*-algebras, particularly in the work related to the Elliott classification program. As an invariant, it plays a significant role in the theory of both simple and non-simple C*-algebras (see \cite{toms}, \cite{c-e}, \cite{ciuperca}, \cite{santiago}). In this paper, we shall show that the Cuntz semigroup is effective as an invariant for the class of C*-algebras that can be expressed as the inductive limit of a sequence
\[
A_1\rightarrow A_2\rightarrow A_3\rightarrow\cdots
\]
where each building block $A_i$ is a separable continuous-trace C*-algebra with spectrum homeomorphic to a disjoint union of trees and trees with a point removed. The term tree will refer to a contractible one-dimensional finite CW complex, in other words, a contractible space obtained from a finite discrete space $\mathrm{V}$ whose elements we shall call vertices, by attaching a finite collection $\mathrm{E}$ of $1$-cells, which we shall call edges. Without loss of generality we may assume that a tree is a subspace of the plane.

Our work can be viewed as a continuation of a number of previous investigations in the classification of C*-algebras.  The problem of classifying inductive limits of continuous trace C*-algebras with spectrum homeomorphic to the closed interval $[0,1]$  was addressed, and results were obtained, in \cite{e-ivan} in the simple case, and in \cite{c-e} in the not necessarily simple case.

A particular class of C*-algebras classified in our paper are the C*-algebras obtained as sequential inductive limits of building blocks of the form $\bigoplus_{k=1}^{N}\M_{m_k}(\mathrm{C}_0(X_k))$, where the spaces $X_k$ are trees or trees with a point removed. The case that the spaces $X_k$ are compact received considerable attention in the more classical framework of classification of simple approximately homogeneous (AH) algebras, where usage of the classical Elliott invariant led to remarkable results. The first such result, the classification of approximate interval (AI) algebras, was obtained by one of the present authors in \cite{elliott-AI}, where the situation $X_k=[0,1]$ was treated. Important generalizations of this result were obtained, in two different directions. On the one hand, the requirement of simplicity for the inductive limit was kept in \cite{Li}, where Li extended the classification to the case where the spaces $X_k$ are trees. On the other hand, classification results for certain classes of non-simple AI algebras were obtained by Stevens in \cite{stevens} and by Robert in \cite{robert}.

In \cite{c-e}, two of the present authors showed that, for C*-algebras of stable rank one, the Cuntz invariant and another C*-algebra invariant---the Thomsen semigroup---determine each other, in a natural way. Using Thomsen's classification result \cite{thomsen}, it was inferred that the Cuntz semigroup is a complete invariant for the class of separable approximate interval algebras. It became apparent that the Cuntz semigroup is a good candidate to be considered in the classification of not necessarily simple C*-algebras.

The main result of the present paper states that the Cuntz semigroup functor classifies the $\ast$-homomorphisms from a  sequential inductive limit of separable continuous-trace C*-algebras with spectrum homeomorphic to a disjoint union of trees and trees with a point removed to a C*-algebra of stable rank one:

\begin{theorem}\label{homomorphism}
Let $A$ be a sequential inductive limit of separable continuous-trace C*-algebras with spectrum homeomorphic to a disjoint union of trees and trees with a point removed. Let $B$ be a C*-algebra of stable rank one. Suppose that there is a Cuntz morphism $\alpha\colon \mathrm{Cu}(A)\to \mathrm{Cu}(B)$ such that $\alpha[s_A]\le [s_B]$, where $s_A$ and $s_B$ are a strictly positive element of $A$ and any positive element of $B$, respectively. It follows that there exists a $\ast$-homomorphism $\phi\colon A\to B$, unique up to approximate unitary equivalence, such that $\mathrm{Cu}(\phi)=\alpha$.
\end{theorem}

It follows from this theorem that the invariant consisting of the Cuntz semigroup together with a distinguished element of it classifies sequential inductive limits of separable continuous-trace C*-algebras with spectrum homeomorphic to a disjoint union of trees and trees with a point removed:
\begin{corollary}\label{classification}
Let $A$ and $B$ be sequential inductive limits of separable continuous-trace C*-algebras with spectrum homeomorphic to a disjoint union of trees and trees with a point removed. Let $s_A$ and $s_B$ be strictly positive elements of $A$ and $B$, respectively. Suppose that there is a Cuntz semigroup isomorphism $\alpha\colon \mathrm{Cu}(A)\to \mathrm{Cu}(B)$ such that $\alpha[s_A]=[s_B]$. It follows that there exists a $\ast$-isomorphism $\phi\colon A\to B$, unique up to approximate unitary equivalence, such that $\mathrm{Cu}(\phi)=\alpha$.
\end{corollary}

Note that, as a consequence of Corollary \ref{classification} and Li's result \cite{Li}, the Cuntz semigroup and the Elliott invariant both classify the simple sequential inductive limits of building blocks of the form $\bigoplus_{k=1}^{N}\M_{m_k}(\mathrm{C}_0(X_k))$, where the spaces $X_k$ are trees. This is not surprising, since for a large variety of simple C*-algebras, the Cuntz semigroup can be recovered functorially from the Elliott invariant, as proved in \cite{b-p-t} and, in the non-unital case, \cite{e-r-s}.

\section{Preliminary definitions and results}
\subsection{The Cuntz semigroup}
For a C*-algebra $A$, let us denote by $A^+$ the set of positive elements of $A$, and by $\widetilde A$ the unitization of $A$. The following definition of the Cuntz semigroup is different from the original definition given in \cite{Cuntz}, in that, in addition to positive elements in matrix algebras over $A$, also positive elements in $A\otimes \K$ are considered ($\K$ denotes the algebra of compact operators on a separable Hilbert space). As shown in \cite{c-e-i}, this form of the definition is very useful. (The relation between the two definitions for a non-stable algebra is not known, except---see \cite{c-e-i}---in the case of stable rank one.)

\begin{definition} Let $a$ and $b$ be positive elements of $A\otimes \K$. Let us say that $a$ is
Cuntz smaller than $b$, denoted by $a\preccurlyeq b$, if there exists a sequence $(d_n)_{n\in \N}$ in $A\otimes \K$ such that
$d_nbd_n^*\to a$. The elements $a$ and $b$ are called Cuntz
equivalent, written $a\sim b$, if $a\preccurlyeq b$ and $b\preccurlyeq a$.
\end{definition}

It is immediate that $\preccurlyeq$ is a pre-order on the set of positive elements of $A\otimes \K$, so that $\sim$ is an equivalence relation. Given $a\in (A\otimes\K)^+$ let us denote by $[a]$ the Cuntz equivalence class of $a$. The Cuntz semigroup of $A$, denoted by $\Cu(A)$, is defined as the set of equivalence classes of positive elements of $A$ endowed with the order derived from the pre-order relation $\preccurlyeq$ (so that $[a]\le [b]$ if $a\preccurlyeq b$), and the addition operation
\begin{align*}
[a] + [b] := \left[\left (\begin{array}{cc} a & 0\\ 0 & b \end{array} \right)\right],
\end{align*}
where the positive element inside the brackets in the right side of the equation above is identified with its image in $A\otimes \K$ by any isomorphism of $\M_2(A\otimes \K)$ with $A\otimes \K$ induced by an isomorphism of $\K$ and $\M_2(\K)$. If $\phi\colon A\to B$ is a $\ast$-homomorphism from the C*-algebra $A$ to the C*-algebra $B$, then it induces an ordered semigroup morphism $\Cu(\phi)\colon \Cu(A)\to \Cu(B)$ defined on the Cuntz equivalence class of a positive element $a\in A\otimes \K$ by $\Cu(\phi)[a]=[(\phi\otimes\id)(a)]$, where $\id\colon \K\to\K$ denotes the identity operator on $\K$.

It was shown in \cite{c-e-i} that $\Cu(A)$ is an object in the the category $\CCu$ of ordered abelian semigroups
with zero with the following additional properties:

(i) every increasing sequence has a supremum;

(ii) each element is the supremum of a rapidly increasing sequence, by which is meant a sequence such that each term is compactly
contained in the next, where we say that an element $x$ is compactly contained in an element $y$, and write
$x\ll y$, if whenever $y\leq \sup y_n$ for an increasing sequence $(y_n)_{n\in \N}$, then eventually $x\leq y_n$;

(iii) the operation of passing to the supremum of an increasing sequence and the relation $\ll$ of compact containment are
compatible with addition.

The maps in the category $\CCu$ are ordered semigroup maps preserving the zero element, suprema of
increasing sequences, and the relation of compact containment. In Theorem 2 of \cite{c-e-i}, the
authors prove that the Cuntz semigroup constitutes a functor from the category of C*-algebras
to the category $\CCu$, and that this functor preserves inductive limits of sequences.

We will also make use of another, stronger, equivalence relation among positive elements of a C*-algebra.

\begin{definition}
Let $a$ and $b$ be positive elements of a C*-algebra $A$. We will say that $a$ is Murray-von Neumann equivalent to $b$ if there exists $x\in A$ such that $a=x^*x$ and $b=xx^*$.
\end{definition}

The following result (the proof of which is included for the convenience of the reader) will be useful on several occasions.

\begin{lemma}\label{ineq} Let $A$ be a C*-algebra. Let $a$ and $b$ be positive elements of $A$ such that $\|a-b\|\le\epsilon$.
Then
\[
\|a^{1/2}-b^{1/2}\|\le\sqrt{\epsilon}.
\]
\end{lemma}

\begin{proof}
Since $\|a-b\|\le\epsilon$, we have
$a\le b+\epsilon 1_{\widetilde{A}}$, where $1_{\widetilde{A}}$ denotes the unit of the unitization of $A$. By Proposition 1.3.8 of \cite{Pedersen} we have
\[
a^{1/2}\le (b+\epsilon 1_{\widetilde{A}})^{1/2}\le b^{1/2}+1_{\widetilde{A}}\sqrt{\epsilon}.
\]
By symmetry,
\[
b^{1/2}\le a^{1/2}+1_{\widetilde{A}}\sqrt{\epsilon}.
\]
Therefore,
\[
\|a^{1/2}-b^{1/2}\|\le\sqrt{\epsilon}.
\]
\end{proof}

The following proposition is a restatement in terms of positive elements of Theorem 3 of \cite{c-e-i}. For the convenience of the reader we include a proof of this statement.

\begin{lemma}\label{x_y}
Let $A$ be C*-algebra, and let $B$ be a hereditary subalgebra of $A$ of stable rank one. Let $\delta>0$. If $x,y\in A$ are such that $xx^*,yy^*\in B$, $x^*x\in \overline{y^*Ay}$, and
\begin{align}\label{xy}
\|x^*x-y^*y\|<\delta,
\end{align}
then there exists a unitary $U$ in the unitization of $B$ such that
\[
\|x-Uy\|<\sqrt{\delta}.
\]

\end{lemma}
\begin{proof}
Let $x=V|x|$ and $y=W|y|$ denote the polar decompositions of $x$ and $y$ in the bidual of $A$ (we use the notation $|x|$ for the element $(x^*x)^\frac{1}{2}$). Set $V|x|W^*=z$. It is clear that $z\in \overline{xA}$. Also, $z\in \overline{Ay^*}$ since by assumption $x^*x\in \overline{y^*Ay}$. Therefore, $z\in B$.

Let $\epsilon>0$ (to be specified later). Since $B$ has stable rank one there exists an invertible element $z'$ in the unitization of $B$ such that
\[
\|z-z'\|\le\frac{\epsilon}{2},\quad \||z|-|z'|\|\le\frac{\epsilon}{2}.
\]
Denote by $U\in \widetilde{B}$ the unitary in the polar decomposition of the invertible element $z'$. It follows that
\[
\|z-U|z|\|\le \|z-z'\|+\|U|z'|-U|z|\|\le\frac{\epsilon}{2}+\frac{\epsilon}{2}=\epsilon.
\]
Hence,
\begin{align}\label{zU}
\|z-U|z|\|<\epsilon.
\end{align}

We have the following estimation:
\begin{align*}
&\|x-Uy\|=\|V|x|-UW|y|\|\le \|V|x|-UW|x|\|+\|UW|x|-UW|y|\|\\
&\le \|V|x|W^*-UW|x|W^*\|+\||x|-|y|\|=\|z-U|z|\|+\||x|-|y|\|.
\end{align*}
From Equation \eqref{xy} we have $\||x|-|y|\|<\sqrt{\delta}$. Taking $\epsilon=\sqrt{\delta}-\||x|-|y|\|$ in Equation \eqref{zU}  and using the preceding estimation we conclude that
\[
\|x-Uy\|\le \|z-U|z|\|+\||x|-|y|\|=\sqrt{\delta}.
\]
\end{proof}

\begin{proposition}\label{prelimprop}
Let $A$ be C*-algebra. If $a,b\in A^+$ are such that $a\preccurlyeq b$, and the hereditary subalgebra $\overline{bAb}$ has stable rank one, then $a$ is Murray-von Neumann equivalent to an element of $\overline{bAb}$.
\end{proposition}
\begin{proof}
Since $a\preccurlyeq b$, it follows by Lemma 2.2 of \cite{k-r} that there are elements $x_n\in A$, $n=1,2,\cdots$, such that
\[
\left(a-\frac{1}{2^{2n}}\right)_+=x_n^*x_n, \quad x_nx_n^*\in \overline{bAb}.
\]
(Given a positive element $a$ and a real number $t$, we denote by $(a-t)_+$ the evaluation of $a$ in the function $c_t(x)=\max(x-t,0)$, for $x\ge 0$.)

For each $n\ge 1$, let us apply Lemma \ref{x_y} to the elements $x_n$ and $x_{n+1}$. Then there exists a unitary $U_n$ such that $\|x_n-U_nx_{n+1}\|<\frac{1}{2^n}$. For each $n\ge 1$ set $U_1U_2\cdots U_{n-1}x_n=x_n'$. We have
\[
(x_n')^*x_n'=\left(a-\frac{1}{2^{2n}}\right)_+, \quad x_n'(x_n')^*\in \overline{bAb},
\]
and $\|x_n'-x_{n+1}'\|<\frac{1}{2^n}$. It follows that the sequence $(x_n')_{n\in \N}$ is a Cauchy sequence and hence it has a limit. Let us denote by $x$ the limit of $(x'_n)_{n\in \N}$. Then
\[
x^*x=\lim_n(x_n')^*x_n'=\lim_n \left(a-\frac{1}{2^{2n}}\right)_+=a.
\]
Also, $xx^*=\lim_n x_n'(x_n')^*\in \overline{bAb}$.
\end{proof}

\subsection{Generators and relations}\label{g-r} Throughout this paper we will only consider trees that are realized as subsets of $\C$, with edges line segments of length $1$. Note that any tree is homeomorphic to one in this class and so we do not lose any generality with this assumption.

Let $(X,v)$ be a rooted tree, that is to say, a tree $X$ with a specified vertex, or root, $v$ and with edges oriented with the natural orientation, {\it away} from the root. Let us denote by $\E(X,v)$ the set of edges of $X$ with their respective orientations. Given two edges $e_1$ of $e_2$ of $(X,v)$, we will say that $e_2$ is {\it next} to $e_1$ if the terminal vertex of $e_1$ is the same as the initial vertex of $e_2$. We will say that $e_1$ is {\it beside} $e_2$ if the initial vertex of $e_1$ is the same as the initial vertex of $e_2$. The orientation of the edges of $X$ induces an order in the set of edges $\E(X,v)$. Given two edges $e$ and $e'$, let us write $e\le e'$ if there is a sequence of edges $e_1,e_2,\cdots,e_n$ with $e_1=e$ and $e_n=e'$ such that $e_{i+1}$ is next to $e_i$ for all $i$.

Let us consider the C*-algebra $\mathrm{C}_0(X\setminus v)$ of continuous functions on $X\subseteq \C$ that vanish at the point $v$. To each edge $e$ of $(X,v)$ we associate the positive element $g_{e}$ of $\mathrm{C}_0(X\setminus v)$ given by
\begin{align}\label{generators}
g_e(x)=
\left\{
\begin{array}{lll}
x, & x\in e=[0,1],\\
1, & x\in X_e,\\
0, & x\in X\setminus (X_e\cup e),
\end{array}
\right.
\end{align}
where $X_e$ denotes the subtree of $X$ consisting of the edges $e'$ of $(X,v)$ that are less than the edge $e$. In the equation \eqref{generators} we are identifying the edge $e$ with the interval $[0,1]$ in such a way that the initial and terminal points of $e$ are identified with the points $0$ and $1$, respectively.  Let us denote by $\mathrm{G}(X,v)$ the set of elements $g_e$ with ${e\in\E(X,v)}$. In the proposition below it is shown that $\mathrm{G}(X,v)$ generates the C*-algebra $\mathrm{C}_0(X\setminus v)$.

Given a rooted tree $(X,v)$, let us denote by $\mathrm{C}^*\langle X,v\rangle$ the universal C*-algebra  on generators $(h_e)_{e\in \E(X,v)}$---one generator $h_e$ for each edge $e$---subject to the relations
\begin{equation}
\begin{aligned}\label{relations}
&h_e\ge 0, \quad \|h_e\|\le 1,\\
& h_{e_1}h_{e_2}=0, \quad\text{ if } e_1 \text{ is beside } e_2,\\
& h_{e_1}h_{e_2}=h_{e_2}, \quad\text{ if } e_2 \text{ is next to } e_1.
\end{aligned}
\end{equation}
Sometimes, in order to avoid confusion, we will write $h^{(X,v)}_e$ instead of $h_e$ for the generators of the C*-algebra $\mathrm{C}^*\langle X,v\rangle$. The same notation will be used when referring to the elements $g_e$ of $\mathrm{C}_0(X\setminus v)$.

\begin{proposition}\label{CC}
Let $(X,v)$ be a rooted tree. Then the C*-algebra $\mathrm{C}^*\langle X,v\rangle$ is isomorphic to the C*-algebra $\mathrm{C}_0(X\setminus v)$, by means of
\begin{equation}\label{h_g}
h_e\mapsto g_e\in \mathrm{C}_0(X\setminus v).
\end{equation}
\end{proposition}
\begin{proof}
The set of elements $(g_e)_{e\in \E(X,v)}$, with $g_e$ defined in Equation \eqref{generators}, is a represention of the relations \eqref{relations} in the C*-algebra $\mathrm{C}_0(X\setminus v)$. It follows by Lemma 3.2.2 of \cite{Loring} that there exists an isomorphism from $\mathrm{C}^*\langle X,v\rangle$ to $\mathrm{C}_0(X\setminus v)$ such that \eqref{h_g} holds, if and only if: (1) the elements $g_e$, $e\in \E(X,v)$, generate the C*-algebra $\mathrm{C}_0(X\setminus v)$; and (2) for any $\ast$-homomorphism $\phi\colon \mathrm{C}^*\langle X,v\rangle\to \C$ there exists a $\ast$-homomorphism $\psi\colon \mathrm{C}_0(X\setminus v)\to \C$ such that $\phi(h_e)=\psi(g_e)$. (Here we are also using that the C*-algebra $\mathrm{C}^*\langle X,v\rangle$ is commutative, since by the relations \eqref{relations} the elements $h_e$, $e\in \E(X,v)$, commute with each other.)

Let $e'$ be an edge of $(X,v)$. The sub-C*-algebra of $\mathrm{C}_0(X\setminus v)$ generated by the element $g_{e'}$ consists of the continuous functions on $X$ that are constant on the set $\bigcup_{e>e'}e$, and zero on the set $X\setminus (\bigcup_{e\ge e'}e)$. These functions, when $e'$ varies through the set $\E(X,v)$, clearly generate the C*-algebra $\mathrm{C}_0(X\setminus v)$. This shows that Condition (1) is satisfied.

Let $\phi\colon \mathrm{C}^*\langle X,v\rangle\to \C$ be a $\ast$-homomorphism. Then the numbers $\phi(h_e)$, $e\in \E(X,v)$, satisfy the relations \eqref{relations}. It follows from these relations that the set of edges of $(X,v)$ such that $\phi(h_e)\neq 0$ consists of a sequence of edges $e_1, e_2,\cdots, e_k$, such that the initial vertex of $e_1$ is $v$, and $e_{i+1}$ is next to $e_i$ for all $i$. Moreover, we have $\phi(h_{e_k})>0$, and $\phi(h_{e_i})=1$ for $1\le i<k$. Let $x$ be the point in $e_k$ such that $g_{e_k}(x)=\phi(h_{e_k})$. Then the $\ast$-homomorphism $\psi\colon \mathrm{C}_0(X\setminus v)\to \C$ such that $\psi(f)=f(x)$ satisfies $\psi(g_e)=\phi(h_e)$ for all $e\in \E(X,v)$. This shows that Condition (2) is satisfied.
\end{proof}

\subsection{Continuous-trace C*-algebras} Let $A$ be a C*-algebra and let $\hat{A}$ denote the spectrum of $A$. A continuous-trace C*-algebra is a C*-algebra $A$ which is generated as a closed two-sided ideal by the elements $x\in A^+$ for which the function $\pi \to \mathrm{Tr}(\pi(x))$ is finite and continuous on $\hat{A}$.

In this paper we make use of the following fact about continuous-trace C*-algebras (see \cite{Dixmier}):
\begin{proposition}\label{c-t}
Let $A$ be a continuous-trace C*-algebra such that $\mathrm{H}^3(\hat{A}, \Z)=0$. Then $A$ is stably isomorphic to $\mathrm{C}_0(\hat{A})$.
\end{proposition}
In particular, the proposition above can be applied to the case that the spectrum of $A$ is a finite disjoint union of trees and trees with a point removed.

\section{The pseudometrics $d_U^{(X,v)}$ and $d_W^{(X,v)}$}
Given a C*-algebra $A$ and a rooted tree $(X,v)$, let us denote by $\Hom(\mathrm{C}_0(X\setminus v), A)$ the set of $\ast$-homomorphisms from $\mathrm{C}_0(X\setminus v)$ to $A$, and by $\Mor(\Cu(\mathrm{C}_0(X\setminus v)),\Cu(A))$ the set of Cuntz semigroup morphisms from the Cuntz semigroup of $\mathrm{C}_0(X\setminus v)$ to the Cuntz semigroup of $A$.  In this section we will define pseudometrics $d_U^{(X,v)}$ and $d_W^{(X,v)}$ on these sets, and in Theorem \ref{pseu-rel} we will prove that these pseudometrics are equivalent.

Given $\phi,\psi\in \Hom(\mathrm{C}_0(X\setminus v),A)$ we define $d^{(X,v)}_U(\phi,\psi)$ by the formula
\begin{align*}
d_U^{(X,v)}(\phi,\psi):=\inf_{U\in \widetilde A}\sup_{g\in \mathrm{G}(X,v)}\|\phi(g)-U^*\psi(g)U\|,
\end{align*}
where $\widetilde A$ denotes the unitization of $A$, and $\mathrm{G}(X,v)$ denotes the set of generators of the C*-algebra $\mathrm{C}_0(X\setminus v)$ corresponding to the edges of $X$ as in \eqref{generators}.

In order to define the pseudometric $d_W^{(X,v)}$ let us consider first the special case $(X,v)=([0,1],0)$. The pseudometric $d^{([0,1],0)}_W$---or, for short, $d_W$---was defined in \cite{c-e} by Ciuperca and Elliott. Given Cuntz semigroup morphisms $\alpha,\beta\colon\Cu(\mathrm{C}_0(0,1])\to \Cu(A)$ the distance between $\alpha$ and $\beta$ is defined by
\begin{align}\label{defdWmor}
d_W(\alpha,\beta):=\inf\left\{r\in \R^+\left|
\begin{array}{c}
\alpha[(\id-t-r)_+]\le\beta[(\id-t)_+],\\
\beta[(\id-t-r)_+]\le\alpha[(\id-t)_+],
\end{array}
\hbox{ for all }t\in \R^+\right\},
\right.
\end{align}
where $\id$ denotes the identity function on $(0,1]$, and $(\id-t)_+$ denotes the positive part of the function $\id-t$, for $t\in \R$.

Now let us consider a general rooted connected tree $(X, v)$. For each generator $g\in \mathrm{G}(X,v)$ let $\chi_g\colon \mathrm{C}_0(0,1]\to \mathrm{C}_0(X\setminus v)$ denote the $\ast$-homomorphism such that $\chi_g(\id)=g$. Let us define $d_W^{(X,v)}$ on the set $\Mor(\Cu(\mathrm{C}_0(X\setminus v)),\Cu(A))$ by
\begin{align}\label{md_X}
d_W^{(X,v)}(\alpha,\beta):=\sup_{g\in \mathrm{G}(X,v)}d_W(\alpha\circ\Cu(\chi_g), \beta\circ \Cu(\chi_g)).
\end{align}
Since $d_W$ is a pseudometric (see \cite{c-e}), $d_W^{(X,v)}$ is also a pseudometric.
In the following proposition we show that when the C*-algebra $A$ has stable rank one $d_W^{(X,v)}$ is actually a metric.
\begin{proposition}\label{metrics}
Let $A$ be a C*-algebra and let $(X,v)$ be a rooted tree. The following statements hold:

(i) If the C*-algebra $A$ has stable rank one, then $d_W^{(X,v)}$ is a metric.

(ii) The space $\Hom(\mathrm{C}_0(X\setminus v), A)$ is complete with respect to the pseudometric $d_U^{(X,v)}$.
\end{proposition}
\begin{proof}
(i) Let $(X,v)$ be a rooted tree, and let $\alpha, \beta\colon \Cu(\mathrm{C}_0(X\setminus v))\to \Cu(A)$ be Cuntz semigroup morphisms such that $d_W^{(X,v)}(\alpha,\beta)=0$. We need to show that $\alpha=\beta$.

By Theorem 1 of \cite{leonel}, $\Cu(\mathrm{C}_0(X\setminus v))$ is naturally isomorphic (by a Cuntz semigroup isomorphism!) to the semigroup of lower semicontinuous functions from $X\setminus v$ to $\N\cup \{\infty\}$ with the pointwise addition and order. The isomorphism is given by the rank map:
\begin{align}\label{rank}
 \Cu(\mathrm{C}_0(X\setminus v))\ni [a] \mapsto \ \{(X\setminus v)\ni x \mapsto\mathrm{rank}(a(x))\}.
\end{align}
Let us denote the set of lower semicontinuous functions from $X\setminus v$ to $\N\cup \{\infty\}$ by $\mathrm{Lsc}(X\setminus v, \N\cup \{\infty\})$. Let $f\in \mathrm{Lsc}(X\setminus v, \N\cup \{\infty\})$. Then $f=\sum_{i=0}^\infty \mathds{1}_{U_i}$, where the sets $U_i$ are defined as $U_i=\{x\mid f(x)>i\}$, for all $i$. Moreover, the sets $U_i$, $i=0,1, \cdots$, are open since $f$ is lower semicontinuous. For each open set $U_i$, $i=1,2,\cdots$, there exists a sequence $(U_{i,j})_{j=1}^{\infty}$ of pairwise disjoint open connected subsets of $X\setminus v$ such that $U_i=\bigcup_{j=1}^\infty U_{i,j}$. Since
\[
f=\sum_{i=0}^\infty \mathds{1}_{U_i}=\sum_{i=0}^\infty\sum_{j=0}^\infty \mathds{1}_{U_{i,j}},
\]
it follows that $f=\sup_{n}\sum_{1\le i,j\le n}\mathds{1}_{U_{i,j}}$.

Let us show that $\alpha(\mathds{1}_U)=\beta(\mathds{1}_U)$ for any open connected subset $U$ of $X\setminus v$. It will follow from this that $\alpha=\beta$, since the Cuntz semigroup morphisms $\alpha$ and $\beta$ are additive and preserve suprema of increasing sequences, and any function in $\mathrm{Lsc}(X\setminus v, \N\cup \{\infty\})$ is the supremum of an increasing sequence of finite sums of characteristic functions of open subsets of $X\setminus v$, as shown above. By assumption, $d_W^{(X,v)}(\alpha, \beta)=0$. Hence, from the definition of the pseudometric $d_W^{(X,v)}$ (see Equation \eqref{md_X})
\[
d_W(\alpha\circ\Cu(\chi_g), \beta\circ \Cu(\chi_g))=0,
\]
for all $g\in \mathrm{G}(X,v)$, where $\chi_g\colon \mathrm{C}_0(0,1]\to \mathrm{C}_0(X\setminus v)$ is the $\ast$-homomorphism defined by $\chi(\id)=g$.
Since the C*-algebra $A$ has stable rank one $d_W$ is a metric, as it was shown in \cite{c-e} (or in Proposition 2 of \cite{l-l}).
Therefore, $\alpha\circ\Cu(\chi_g)=\beta\circ \Cu(\chi_g)$ for all $g\in \mathrm{G}(X,v)$. In particular,
$\alpha[(g_e-\epsilon)_+]=\beta[(g_e-\epsilon)_+]$ for each $e\in \E(X,v)$, and each $\epsilon>0$.
Let us denote by $X_e^\epsilon$ the subset of $X\setminus v$ consisting of the edges of $(X,v)$ that are less than $e$,
and the points of $e$ that are at a distance strictly larger than $\epsilon$ from the initial vertex of $e$.
Then $\mathrm{rank}((g_e-\epsilon)_+)=\mathds{1}_{X_e^\epsilon}$. It follows that
$\alpha(\mathds{1}_{X_e^\epsilon})=\beta(\mathds{1}_{X_e^\epsilon})$ for each $e\in \mathrm{E}(X,v)$, and each $\epsilon>0$.
Note that each set $X_e^\epsilon$ is hereditary in the sense that, if $x\in X_e^\epsilon$
and $y<x$ (a point $y$ is less than a point $x$ if the non-overlapping path from $y$ to the root $v$ contains $x$), then $y\in X_e^\epsilon$.
Also note that every hereditary open subset of $X\setminus v$ is the union of a finite number of pairwise disjoint sets that have the form
$X_e^\epsilon$, for some $e\in \mathrm{E}(X,v)$ and some $\epsilon>0$. Therefore, $\alpha(\mathds{1}_U)=\beta(\mathds{1}_U)$ for every hereditary
open subset $U$ of $X\setminus v$.

Let $U$ be a connected open subset of $X\setminus v$. We can choose a sequence of open subsets $(U_i)_{i=1}^\infty$ such that
$U=\bigcup_{i=1}^\infty U_i$, and such that for each $i=1,2,\cdots$, there is a compact set $K_i$ such that $U_i\subset K_i\subset U_{i+1}$.
It follows that $\mathds{1}_U=\sup_i \mathds{1}_{U_i}$, and $\mathds{1}_{U_i}\ll \mathds{1}_{U_{i+1}}$ for all $i\ge 1$. In addition, for each
$i=1,2,\cdots$, we can chose a hereditary open subset $V_i\subseteq X\setminus v$ such that $U\cup V_i$ is hereditary, $V_i\cup K_i$ is compact,
and $V_i\cap U_i=\varnothing$. Hence, $V_i$ is such that
\begin{align*}
&\alpha(\mathds{1}_{V_i})=\beta(\mathds{1}_{V_i}),\quad \alpha(\mathds{1}_{U\cup V_i})=\beta(\mathds{1}_{U\cup V_i}),\\
&\mathds{1}_{U_i\cup V_i}\ll \mathds{1}_{U_{i+1}\cup V_i},\quad \mathds{1}_{U_i}+\mathds{1}_{V_i}=\mathds{1}_{U_i\cup V_i}.
\end{align*}
We have
\begin{align*}
\alpha(\mathds{1}_{U_i})+\alpha(\mathds{1}_{V_i})&=\alpha(\mathds{1}_{U_i\cup V_i})\\
&\ll \alpha(\mathds{1}_{U_{i+1}\cup V_i})\\
&\le \alpha(\mathds{1}_{U\cup V_i})=\beta(\mathds{1}_{U\cup V_i})\\
&\le \beta(\mathds{1}_{U})+\beta(\mathds{1}_{V_i})\\
&=\beta(\mathds{1}_{U})+\alpha(\mathds{1}_{V_i}).
\end{align*}
More briefly,
\[
\alpha(\mathds{1}_{U_i})+\alpha(\mathds{1}_{V_i})\ll \beta(\mathds{1}_{U})+\alpha(\mathds{1}_{V_i}).
\]
Hence by Theorem 4.3 of \cite{r-w},
\[
\alpha(\mathds{1}_{U_i})\le \beta(\mathds{1}_{U}).
\]
Taking the supremum over $i\ge 1$ and using that $\mathds{1}_U=\sup_i \mathds{1}_{U_i}$ we conclude that $\alpha(\mathds{1}_{U})\le \beta(\mathds{1}_{U})$.
By symmetry, $\beta(\mathds{1}_{U})\le \alpha(\mathds{1}_{U})$. Therefore, we have $\alpha(\mathds{1}_{U})=\beta(\mathds{1}_{U})$ for every open connected
subset of $X\setminus v$. This concludes the proof of
the statement (i).

(ii) Let $(\phi_n)_{n\in \N}$ be a Cauchy sequence with respect to the pseudometric $d^{(X,v)}_U$. Then there exists a subsequence $(\phi_{n_i})_{i\in \N}$ of $(\phi_n)_{n\in \N}$ such that
\begin{align}\label{phi_ni_n}
d_U^{(X,v)}(\phi_{n_i},\phi_{n_{i+1}})<\frac{1}{2^{i+1}}, \quad d_U^{(X,v)}(\phi_{n_i},\phi_n)<\frac{1}{2^{i+1}},
\end{align}
for each $i\in \N$, and each $n>n_i$. It follows from the definition of the metric $d_U^{(X,v)}$, and from the first inequality above, that there exist unitaries $U_i\in \tilde{A}$, $i=1,2,\cdots$, such that
\begin{align*}
\|\phi_{n_i}(g)-U_i^*\phi_{n_{i+1}}(g)U_i\|<\frac{1}{2^{i+1}},
\end{align*}
for all $g\in \mathrm{G}(X,v)$. For each $i\ge 1$, set $\Ad(U_{i-1}U_{i-2}\cdots U_1)\circ \phi_{n_i}=\phi'_{n_i}$. Then for each $g\in \mathrm{G}(X,v)$, the sequence $(\phi'_{n_i}(g))_{i\in \N}$ is Cauchy in the norm topology. Hence, it converges. For each $g\in \mathrm{G}(X,v)$, let us denote by $\hat{g}$ the element $\lim_i\phi_{n_i}(g)$. Then the set $\{\hat{g}\mid g\in \mathrm{G}(X,v) \}$ is a representation of the relations \eqref{relations} in the C*-algebra $A$. Therefore, by Proposition \ref{CC}, there exists a $\ast$-homomorphism $\phi\colon \mathrm{C}_0(X\setminus v)\to A$ such that $\phi(g)=\hat{g}$, for all $g\in \mathrm{G}(X,v)$. Using the triangle inequality and the second inequality in Equation \eqref{phi_ni_n} we have
\begin{align*}
d_U^{(X,v)}(\phi_n, \phi)&\le d_U^{(X,v)}(\phi_n, \phi_{n_i})+d_U^{(X,v)}(\phi_{n_i}, \phi'_{n_i})+d_U^{(X,v)}(\phi'_{n_i}, \phi)\\
 &<\frac{1}{2^{i+1}}+0+d_U^{(X,v)}(\phi'_{n_i}, \phi)=\frac{1}{2^{i+1}}+d_U^{(X,v)}(\phi'_{n_i}, \phi),
\end{align*}
for each $i\in \N$, and for each $n>n_i$. It follows that $d_U^{(X,v)}(\phi_n, \phi)\to 0$ when $n\to \infty$. Therefore, the pseudometric $d_U^{(X,v)}$ is complete.
\end{proof}

\begin{lemma} \label{lemma-metrics}
Let $A$ be a unital C*-algebra of stable rank one, and let $X$ be a tree. Consider two Cuntz semigroup morphisms $\alpha, \beta\colon \Cu(\mathrm{C}(X))\to \Cu(A)$  such that $\alpha([1_X])=\beta([1_X])$, where $1_X$ denotes the unit of $\mathrm{C}(X)$. Then
\begin{align}\label{identity1}
d_W(\alpha\circ\Cu(\chi_g), \beta\circ \Cu(\chi_g))=d_W(\alpha\circ\Cu(\chi_{1-g}), \beta\circ \Cu(\chi_{1-g})),
\end{align}
for any element $0\le g\le 1$ of $\mathrm{C}(X)$, where $\chi_g$ as above takes $\id \in \mathrm{C}_0(0,1]$ to $g$.
\end{lemma}
\begin{proof}
Let $0\le g\le 1$ in $\mathrm{C}(X)$ be given. Let us denote by $\widetilde{\chi}_g$ the $\ast$-homomorphism from $\mathrm{C}[0,1]$ to $\mathrm{C}(X)$ such that $\widetilde{\chi}_g(\id)=g$ and $\widetilde{\chi}_g(1_{[0,1]})=1_X$. Note that $\widetilde{\chi}_g|_{\mathrm{C}_0(0,1]}=\chi_g$ and that, in the obvious sense, $\widetilde{\chi}_g|_{C_0[0,1)}=\chi_{1-g}$. Using the definition of the pseudometrics $d_W=d_W^{([0,1],0)}$ and $d_W^{([0,1],1)}$ we have
\begin{align*}
d_W(\alpha\circ\Cu(\chi_g), \beta\circ \Cu(\chi_g))&=d_W(\alpha\circ\Cu(\widetilde\chi_g), \beta\circ \Cu(\widetilde\chi_g)),\\
d_W(\alpha\circ\Cu(\chi_{1-g}), \beta\circ \Cu(\chi_{1-g}))&=d_W^{([0,1],1)}(\alpha\circ\Cu(\widetilde\chi_g), \beta\circ \Cu(\widetilde\chi_g)).
\end{align*}
Therefore, in order to show that equality \eqref{identity1} holds it is enough to show that $d_W(\alpha,\beta)=d^{([0,1],1)}_W(\alpha,\beta)$ for all Cuntz semigroup morphisms $\alpha,\beta\colon \Cu(\mathrm{C}[0,1])\to \Cu(A)$ such that $\alpha([1_{[0,1]}])=\beta([1_{[0,1]}])$.

By Theorem 1 of \cite{leonel} the Cuntz semigroup of $\mathrm{C}[0,1]$ is isomorphic to the set of lower semicontinuous functions from $[0,1]$ to $\N\cup \{\infty\}$ with pointwise addition and order. Under this identification the element $[(\id-t)_+]\in\Cu(\mathrm{C}[0,1])$ corresponds to the characteristic function of the set $(t,1]$. We will denote this function by $\mathds{1}_{(t,1]}$. Let $r>0$. It follows from the definition of $d_W$ and $d_W^{([0,1],1)}$ that
\begin{align}
d_W(\alpha,\beta)<r \Leftrightarrow
\begin{array}{c}
\alpha(\mathds{1}_{(t+r,1]})\le\beta(\mathds{1}_{(t,1]}),\\
\beta(\mathds{1}_{(t+r,1]})\le\alpha(\mathds{1}_{(t,1]}),
\end{array}
\hbox{ for all }t\in \R^+,\label{eq0}\\
d_W^{([0,1],1)}(\alpha,\beta)<r \Leftrightarrow
\begin{array}{c}
\alpha(\mathds{1}_{[0, t)})\le\beta(\mathds{1}_{[0, t+r)}),\\
\beta(\mathds{1}_{[0, t)})\le\alpha(\mathds{1}_{[0, t+r)}),
\end{array}
\hbox{ for all }t\in \R^+.\label{eq1}
\end{align}
In order to show that $d_W(\alpha,\beta)=d_W^{([0,1],1)}(\alpha,\beta)$ it is enough to prove that $d_W(\alpha,\beta)<r$ implies $d_W^{([0,1],1)}(\alpha,\beta)<r$ for all $r\in \R^+$, and vice versa.

Let us suppose that $d_W(\alpha,\beta)<r$. Then for all $\epsilon>0$ we have (somewhat as in \cite{c-e})
\begin{align*}
\alpha(\mathds{1}_{[0,t-\epsilon)})+\alpha(\mathds{1}_{(t-\epsilon,1]})&\le\alpha(\mathds{1}_{[0,1]})\ll\alpha(\mathds{1}_{[0,1]})=\beta(\mathds{1}_{[0,1]})\\
&\le \beta(\mathds{1}_{[0,t+r)})+\beta(\mathds{1}_{(t+r-\epsilon,1]})\\
&\le \beta(\mathds{1}_{[0,t+r)})+\alpha(\mathds{1}_{(t-\epsilon,1]})
\end{align*}
(we are using relation \eqref{eq0} in order to obtain the last inequality above). By Theorem 1 of \cite{elliott-cancellation} (or by Theorem 4.3 of \cite{r-w}) we can cancel $\alpha(\mathds{1}_{(t-\epsilon,1]})$ from both sides of the preceding inequality. Thus, we obtain $\alpha(\mathds{1}_{[0,t-\epsilon)})\le \beta(\mathds{1}_{[0,t+r)})$. Since $\epsilon$ is arbitrary and $\alpha$
preserves suprema of increasing sequences,
\[
\alpha(\mathds{1}_{[0,t)})=\alpha(\sup_\epsilon\mathds{1}_{[0,t-\epsilon)})=\sup_\epsilon\alpha(\mathds{1}_{[0,t-\epsilon)})\le \beta(\mathds{1}_{[0,t+r)}).
\]
So, $\alpha(\mathds{1}_{[0,t)})\le \beta(\mathds{1}_{[0,t+r)})$. Interchanging $\alpha$ and $\beta$, as we may, we have $\beta(\mathds{1}_{[0,t)})\le \alpha(\mathds{1}_{[0,t+r)})$. Hence, $d_W^{([0,1],1)}(\alpha,\beta)<r$ by relation \eqref{eq1}. This shows that
\[
d_W^{([0,1],1)}(\alpha,\beta)\le d_W(\alpha,\beta).
\]
The opposite inequality follows by symmetry.
\end{proof}

\begin{theorem}\label{m-rel}
Let $A$ be a C*-algebra of stable rank one and let $(X,v)$ be a rooted tree. Let $\phi, \psi\colon \mathrm{C}_0(X\setminus v)\to A$ be $\ast$-homomorphisms. Then for all $\epsilon>0$ there exists a unitary $U$ in $\widetilde A$ such that
\begin{align}\label{result}
\|\phi(g)-U^*\psi(g)U\|< (2N+2)d_W(\Cu(\phi\circ \chi_g), \Cu(\psi\circ \chi_g))+\epsilon
\end{align}
for all $g\in \mathrm{G}(X,v)$, where $N$ denotes the number of edges of $X$.
\end{theorem}
\begin{proof}
We will use mathematical induction on the number of edges $N$ of the tree $X$. When $N=1$ we can identify the rooted tree $(X,v)$ with $([0,1],0)$. In this case the inequality \eqref{result} was shown in the proof of Theorem 4.1 of \cite{c-e} (for an explicit proof see Theorem 1 of \cite{l-l}).

Now let $N>1$ and let us assume that the theorem holds for all rooted trees with number of edges strictly less than $N$, and in each case for an arbitrary C*-algebra of stable rank one in place of $A$. Let us show that the theorem holds for all rooted trees with $N$ edges. Let $(X,v)$ be a rooted tree with $N$ edges and let $\phi,\psi\colon \mathrm{C}_0(X\setminus v)\to A$ be $\ast$-homomorphisms.  Let us denote by $\widetilde\phi, \widetilde\psi\colon \mathrm{C}(X)\to \widetilde{A}$ the unitizations of $\phi$ and $\psi$. Then $\phi\circ\chi_g=\widetilde\phi\circ\chi_g$ and $\psi\circ\chi_g=\widetilde\psi\circ\chi_g$. It follows from these identities that the inequality \eqref{result} holds for the $\ast$-homomorphisms $\phi$ and $\psi$ if and only if it holds for the $\ast$-homomorphisms $\widetilde\phi$ and $\widetilde\psi$. Therefore, we may assume that the $\ast$-homomorphisms $\phi$ and $\psi$ have domain $\mathrm{C}(X)$, codomain $\widetilde{A}$, and that they are unital.

Let us show that there is no loss of generality to assume that the number of edges of $X$ having $v$ as a vertex is strictly larger than one. Assume that $v$ is a vertex of only one edge of $X$, say $e$. Denote by $v'$ the other vertex of the edge $e$. Since the number of edges of the tree $X$ is strictly larger than one, $v'$ is a vertex of more than one edge of $X$.  From the definition of the set of generators $\mathrm{G}(X,v')$ we see that it is obtained from the set $\mathrm{G}(X,v)$ by replacing the generator $g_e$ associated to the edge $e$ by the function $1-g_e\in \mathrm{C}_0(X\setminus v')$. That is to say,
\begin{align}\label{ch-vertices}
\mathrm{G}(X,v')=(\mathrm{G}(X,v)\setminus\{g_e\})\cup\{1-g_e\}.
\end{align}
Now if we apply Lemma \ref{lemma-metrics} to the C*-algebra $\widetilde{A}$ and to the Cuntz semigroup
morphisms $\Cu(\phi)$ and $\Cu(\psi)$ we see that the inequality \eqref{result} remains unchanged if a generator $g$ is replaced by $1-g$.
This observation together with the equation \eqref{ch-vertices} implies that the $\ast$-homomorphisms $\phi$
and $\psi$ satisfy the inequality \eqref{result} for all $g$ in $\mathrm{G}(X,v)$ if and only if they satisfy
this condition with $v'$ in place of $v$. Thus, we may assume that $v$ is a vertex of more than one edge of $X$.
Let us denote by $e_1,e_2,\cdots,e_k$ the edges of $(X,v)$ with $v$ as a vertex, and by $v_1,v_2,\cdots,v_k$ their second vertices, respectively. Denote
by $g_{e_1},g_{e_2},\cdots,g_{e_k}\in \mathrm{G}(X,v)$ the generators of $\mathrm{C}_0(X\setminus v)$
associated to $e_1,e_2,\cdots,e_k$. Suppose that the number of indices $i$ such that
\begin{align*}
d_W(\Cu(\phi\circ \chi_{g_{e_i}}), \Cu(\psi\circ \chi_{g_{e_i}}))= 1,
\end{align*}
is strictly larger than one (note that by the definition of the metric $d_W$ we have $d_W(\alpha,\beta)\le 1$ for all
Cuntz semigroup morphisms $\alpha$ and $\beta$). Changing the numbering of the edges $e_1,e_2,\cdots,e_k$ if necessary, we may assume that
\begin{align}
&d_W(\Cu(\phi\circ \chi_{g_{e_i}}), \Cu(\psi\circ \chi_{g_{e_i}}))=1,\text{ for }1\le i\le k',\label{equalone}\\
&d_W(\Cu(\phi\circ \chi_{g_{e_i}}), \Cu(\psi\circ \chi_{g_{e_i}}))<1,\text{ for }k'< i\le k.\nonumber
\end{align}
Denote by $Y$ the subgraph of $X$ consisting of the edges $e$ of $(X,v)$ such that either $e<e_i$ for some $1\le i\le k$, or $e=e_i$ for some $k'<i\le k$. Let us define an equivalence relation $\sim$ on $Y$
by taking $v_i\sim v_j$ for every $1\le i,j\le k'$, $v_i\sim v$ for every $1\le i\le k'$, and $x\sim x$ for every $x\in Y$. Then the set $(Y', [v_1])$, where $Y'=Y/\negthickspace\sim$ and $[v_1]$ denotes the equivalence class of $v_1$, has the structure of a rooted tree. The edges and the vertices of $(Y', [v_1])$ are defined to be the images by the quotient map of the edges and the vertices of $Y$. In addition, the C*-algebra $\mathrm{C}_0(Y'\setminus [v_1])$ is isomorphic to the subalgebra of $\mathrm{C}_0(X\setminus v)$ generated by the elements $g_e\in \mathrm{G}(X,v)$ such that $e$ is an edge of $Y$. Let $\phi',\psi'\colon \mathrm{C}_0(Y'\setminus [v_1])\to \widetilde A$ denote the $\ast$-homomorphisms induced by the restrictions of the $\ast$-homomorphisms $\phi$ and $\psi$ to the subalgebra of $\mathrm{C}_0(X\setminus v)$ generated by the elements $g_e\in \mathrm{G}(X,v)$, with $e$ an edge of $Y$. By the inductive hypothesis and using that the number of edges of $(Y', [v_1])$ is strictly less than $N$ we have that given $\epsilon>0$ there exists a unitary $U\in \widetilde{A}$ such that
\[
\|\phi'(g)-U^*\psi'(g)U\|< (2N+2)d_W(\Cu(\phi'\circ \chi_g), \Cu(\psi'\circ \chi_g))+\epsilon,
\]
for all $g\in \mathrm{G}(Y', [v_1])$. It follows that
\begin{align}\label{ineq_e}
\|\phi(g_e)-U^*\psi(g_e)U\|< (2N+2)d_W(\Cu(\phi\circ \chi_{g_e}), \Cu(\psi\circ \chi_{g_e}))+\epsilon,
\end{align}
for every edge $e$ of $Y$. By Equation \eqref{equalone} the inequality above also holds for the edges $e_1,e_2,\cdots,e_{k'}$ since the left side of the inequality \eqref{ineq_e} is less than or equal to two, and $N>1$. Using that the edges of $(X,v)$ consist of the edges of $Y$ and the edges $e_1,e_2,\cdots,e_{k'}$ we conclude that the inequality above holds for all $e\in \mathrm{E}(X,v)$. It follows that the statement of the theorem holds for the rooted tree $(X,v)$. Therefore, we may assume that
\begin{align*}
d_W(\Cu(\phi\circ \chi_{g_{e_i}}), \Cu(\psi\circ \chi_{g_{e_i}}))<r_i<1,
\end{align*}
for $i=1,2,\cdots,k$ and some positive numbers $r_i$.

By the definition of the pseudometric $d_W$, the preceding inequality implies that
\begin{align*}
\phi((g_{e_i}-r_i)_+)=(\phi(g_{e_i})-r_i)_+\preccurlyeq \psi(g_{e_i}),
\end{align*}
for each $i=1,2,\cdots,k$. Applying (ii) of Proposition \ref{prelimprop} to the elements $\phi((g_{e_i}-r_i)_+)$ and $\psi(g_{e_i})$, we obtain elements $x_i\in \widetilde A$, $i=1,2,\cdots,k$, such that
\begin{align}\label{xi}
\phi((g_{e_i}-r_i)_+)=x_i^*x_i, \quad x_ix_i^*\in \overline{\psi(g_{e_i})\widetilde A\psi(g_{e_i})}.
\end{align}
Note that the elements $x_i$ satisfy the orthogonality relations $x_ix_j^*=x_i^*x_j=0$ for $i\neq j$ (this holds because the elements $g_{e_i}$ are pairwise orthogonal). Set $\sum_{i=1}^kx_i=x$, and consider the polar decomposition $x=V|x|$ of $x$ in the bidual of $\widetilde A$. From the orthogonality relations satisfied by the elements $x_i$ it follows that $x_i=V|x_i|$. This last identity implies that the restriction of the map $\Ad(V^*)\circ \phi$ to the C*-algebra $\overline{(g_{e_i}-r_i)_+\mathrm{C}_0(X\setminus v)}$ is a $\ast$-homomorphism with image contained in the hereditary subalgebra $\overline{\psi(g_{e_i})\widetilde A\psi(g_{e_i})}$ of $\widetilde A$, for each $i=1,2,\cdots,k$.

For each $i\in \{1,2,\cdots,k\}$ let us denote by $X_i$ the closure (in $\C$) of the spectrum of the algebra $\overline{(g_{e_i}-r_i)_+\mathrm{C}_0(X\setminus v)}$, and by $w_i\in X_i$ the point on the edge $e_i$ that is at distance $r_i$ from $v$.  The set $X_i\subset X$ can be given the structure of a tree by defining its vertices to be the vertices of $X$ that belong to $X_i$, together with the point $w_i$. The edges of $X_i$ will simply be the edges of $X$ that are subsets of $X_i$, together with the part of the edge $e_i$ that belongs to $X_i$. (We will refrain from insisting here that an edge has length one.) It follows from the fact that $k$ is at least two that the number of edges of $X_i$ is less than or equal to $N-1$. Note that the C*-algebra $\overline{(g_{e_i}-r_i)_+\mathrm{C}_0(X\setminus v)}$ is just $\mathrm{C}_0(X_i\setminus w_i)$.

The restrictions of the maps $\Ad(V^*)\circ\phi$ and $\psi$ to the C*-algebra $\mathrm{C}_0(X_i\setminus w_i)$ are $\ast$-homomorphisms with images contained in the hereditary subalgebra $\overline{\psi(g_{e_i})\widetilde A\psi(g_{e_i})}$. Therefore, by the inductive hypothesis, for each fixed $\epsilon>0$ there exists a unitary $U_i$, in the C*-algebra generated by the hereditary subalgebra $\overline{\psi(g_{e_i})\widetilde A\psi(g_{e_i})}$ and the unit of $\widetilde A$, such that
\begin{align*}
\| V\phi(g)V^*-U_i^*\psi(g)U_i\|&< 2N d_W(\Cu(\Ad(V^*)\circ\phi\circ \chi_g), \Cu(\psi\circ \chi_g))+\epsilon\\
&=2N d_W(\Cu(\phi\circ \chi_g), \Cu(\psi\circ \chi_g))+\epsilon,
\end{align*}
for all $g\in \mathrm{G}(X_i, w_i)$. Replacing $U_i$ by a scalar multiple if necessary, we may assume that $U_i-1\in \overline{\psi(g_{e_i})\widetilde A\psi(g_{e_i})}$. Set $1+\sum_{i=1}^k(U_i-1)=U$. Then $U$ is a unitary element of $\widetilde A$ since the elements $U_i-1$ are pairwise orthogonal. Furthermore, for each $i=1,2,\cdots,k$, we have $U_i^*\psi(g)U_i=U^*\psi(g)U$ for all $g\in \mathrm{G}(X_i, w_i)$. Thus,
\begin{align}\label{VU}
\| V\phi(g)V^*-U^*\psi(g)U\|< 2N d_W(\Cu(\phi\circ \chi_g), \Cu(\psi\circ \chi_g))+\epsilon,
\end{align}
for all $g\in \bigcup_{i=1}^k \mathrm{G}(X_i, w_i)$.

Recall that $V$ is the partial isometry in the polar decomposition of the element $x=\sum_{i=1}^kx_i$, where the elements $x_i$ are given in the equation \eqref{xi}. It follows by (i) of Proposition \ref{prelimprop} that for any $\delta>0$ there exists a unitary $W\in \widetilde A$ such that
\[
\|V|x|-W|x|\|<\delta.
\]
Hence for any $\delta>0$ and any finite subset $F$ of the hereditary subalgebra $\overline{|x|\widetilde A|x|}$ there exists a unitary $W\in \widetilde A$ such that
\[
\|VyV^*-WyW^*\|<\delta,
\]
for all $y\in F$. In particular, if we take $F=\bigcup_{i=1}^k \mathrm{G}(X_i, w_i)$, and $\delta$ small enough, we find that the inequality  \eqref{VU} still holds if the partial isometry $V$ is replaced by a suitable unitary $W$. Set $UW=U'$. Then
\begin{align}\label{U'}
\| \phi(g)-(U')^*\psi(g)U'\|< 2Nd_W(\Cu(\phi\circ \chi_g), \Cu(\psi\circ \chi_g))+\epsilon,
\end{align}
for all $g\in \bigcup_{i=1}^k \mathrm{G}(X_i, w_i)$. Let us show that the unitary $U'$ satisfies the conditions of the theorem. According to the tree structure given to the set $X_i$ for each $i=1, 2,\cdots, k$, we have that all the elements of $\mathrm{G}(X_i, w_i)$ belong to $\mathrm{G}(X,v)$ except for $f_i=\frac{1}{1-r_i}(g_{e_i}-r_i)_+$ (this function, restricted to $X_i$, is the generator of $\mathrm{C}(X_i, w_i)$ that corresponds to the (short) edge $X_i\cap e_i$; in other words, it is the generator $g_{X_i\cap e_i}$). In fact,
\begin{align}\label{G}
\mathrm{G}(X,v)=\bigcup_{i=1}^k(\mathrm{G}(X_i, v_i)\cup \{g_{e_i}\})\backslash \bigcup_{i=1}^{k}\{f_i\}.
\end{align}
So, in order to show that the inequality \eqref{result} holds for the $\ast$-homomorphisms $\phi$ and $\psi$, and the unitary $U'$, it is enough to check that it holds for the elements $g_{e_i}$. (For the rest of the elements of $\mathrm{G}(X,v)$ it holds by \eqref{U'} and \eqref{G}.) By the definition of the metric $d_W$,
\begin{align*}
d_W(\Cu(\phi\circ \chi_{f_i}), \Cu(\psi\circ \chi_{f_i}))&=\frac{1}{1-r_i}d_W(\Cu(\phi\circ \chi_{(g_{e_i}-r_i)_+}), \Cu(\psi\circ \chi_{(g_{e_i}-r_i)_+}))\\
&\le\frac{1}{1-r_i}d_W(\Cu(\phi\circ \chi_{g_{e_i}}), \Cu(\psi\circ \chi_{g_{e_i}})).
\end{align*}
Since $d_W(\Cu(\phi\circ \chi_{g_{e_i}}), \Cu(\psi\circ \chi_{g_{e_i}}))<r_i$ (by hypothesis), it follows that
\[
d_W(\Cu(\phi\circ \chi_{f_i}), \Cu(\psi\circ \chi_{f_i}))<\frac{r_i}{1-r_i}.
\]
By the inequality \eqref{U'} with $g=f_i$, it follows that
\begin{align*}
\|\phi(f_i)-(U')^*\psi(f_i)U'\|<\frac{2N r_i}{1-r_i}+\epsilon,
\end{align*}
for $i=1,2,\cdots,k$. Hence,
\begin{align}\label{2N}
\| \phi((g_{e_i}-r_i)_+)-(U')^*\psi((g_{e_i}-r_i)_+)U'\|<2N r_i+(1-r_i)\epsilon\le 2Nr_i+\epsilon,
\end{align}
for $i=1,2,\cdots,k$.

By the triangle inequality,
\begin{align*}
&\| \phi(g_{e_i})-(U')^*\psi(g_{e_i})U'\| \le  \|\phi(g_{e_i})-\phi((g_{e_i}-r_i)_+)\|+\\
&+\| \phi((g_{e_i}-r_i)_+)-(U')^*\psi((g_{e_i}-r_i)_+)U'\|
+\|\psi((g_{e_i}-r_i)_+)-\psi(g_{e_i})\|.
\end{align*}
By the inequality \eqref{2N} and the identity $\|(g_{e_i}-r_i)_+-g_{e_i}\|=r_i$ it follows that
\begin{align*}
\| \phi(g_{e_i})-(U')^*\psi(g_{e_i})U'\|< r_i+(2N r_i+\epsilon)+r_i=(2N+2)r_i+\epsilon,
\end{align*}
for $i=1,2,\cdots,k$. Since this inequality holds for all numbers $r_i$ such that
\[
d_W(\Cu(\phi\circ\chi_{g_{e_i}}),\Cu(\psi\circ\chi_{g_{e_i}}))<r_i<1
\]
we conclude that
\begin{align*}
\| \phi(g_{e_i})-(U')^*\psi(g_{e_i})U'\|<(2N+2)d_W(\Cu(\phi\circ\chi_{g_{e_i}}),\Cu(\psi\circ\chi_{g_{e_i}}))+\epsilon,
\end{align*}
for $i=1,2,\cdots,k$. This shows that \eqref{result} holds for the elements $g_{e_i}$, $i=1,2,\cdots,k$. This concludes the proof of the theorem.
\end{proof}

\begin{theorem}\label{pseu-rel}
Let $A$ be a C*-algebra of stable rank one and let $(X,v)$ be a rooted tree. Let $\phi,\psi\colon \mathrm{C}_0(X\setminus v)\to A$ be $\ast$-homomorphisms. Then
\begin{align}\label{metric-rel}
d_W^{(X,v)}(\Cu(\phi),\Cu(\psi))\le d_U^{(X,v)}(\phi,\psi)\le (2N+2)d_W^{(X,v)}(\Cu(\phi),\Cu(\psi)),
\end{align}
where $N$ denotes the number of edges of $X$.
\end{theorem}
\begin{proof}
Let us start by proving the first inequality of \eqref{metric-rel}. By Corollary 9.1 of \cite{c-e} (see also Lemma 1 of \cite{l-l}) we have, for each $g\in \mathrm{G}(X,v)$,
\begin{align*}
d_W(\Cu(\phi\circ\chi_g),\Cu(\psi\circ\chi_g))&\le d_U^{([0,1],0)}(\phi\circ\chi_g, \psi\circ\chi_g)\\
&=\inf_{u\in A\,\tilde{}}\|\phi(g)-u^*\psi(g)u\|\\
&\le \inf_{u\in A\,\tilde{}}\sup_{g\in \mathrm{G}(X,v)}\|\phi(g)-u^*\psi(g)u\|\\
&= d_U^{(X,v)}(\phi,\psi).
\end{align*}
It follows that
\begin{align*}
d_W^{(X,v)}(\Cu(\phi),\Cu(\psi))=\sup_{g\in \mathrm{G}(X,v)}d_W(\Cu(\phi\circ\chi_g),\Cu(\psi\circ\chi_g))\le d_U^{(X,v)}(\phi,\psi).
\end{align*}

Now let us prove the second inequality of \eqref{metric-rel}. Applying Theorem \ref{m-rel} to the $\ast$-homomorphisms $\phi$ and $\psi$ we obtain a unitary $U\in \widetilde A$ such that
\begin{align*}
\|\phi(g)-U^*\psi(g)U\|< (2N+2) d_W(\Cu(\phi\circ \chi_g), \Cu(\psi\circ \chi_g))+\epsilon,
\end{align*}
for all $g\in \mathrm{G}(X,v)$. Taking the suprema on both sides of the inequality above with respect to $g\in \mathrm{G}(X,v)$ and then taking the infimum with respect to $U\in \widetilde A$ we obtain
\[
d_U^{(X,v)}(\phi,\psi)< (2N+2)d_W^{(X,v)}(\Cu(\phi), \Cu(\psi))+\epsilon.
\]
Since $\epsilon$ is arbitrary, the desired inequality follows.
\end{proof}

\section{Approximate lifting}
This section is devoted to the proof of Theorem \ref{a-l} below. This theorem states that every Cuntz semigroup morphism between the Cuntz semigroups of the C*-algebra of continuous functions over a rooted connected tree and a general C*-algebra can be lifted approximately to a $\ast$-homomorphism between these C*-algebras. Before we proceed to prove Theorem \ref{a-l} we need some preliminary results.

The following result is Lemma 4 of \cite{l-l}. (Lemma \ref{interpolation} is used in the proof of Theorem 4 of \cite{l-l} which is the case of Theorem \ref{a-l} below that the tree is an interval. The proof in the present more general setting is somewhat more subtle than the proof on the case of the interval.)

\begin{lemma}\label{interpolation}
Let $A$ be a C*-algebra and let $\{x_k\}_{k=0}^{n}$ be elements of $\Cu(A)$ such that $x_{k+1}\ll x_k$ for $k=0,1,\cdots,n-1$. Then there exists $a\in (A\otimes \K)^+$ with $\|a\|\leq 1$ such that $[a]=x_0$ and $x_{k+1}\ll[(a-k/n)_+]\ll x_k$ for $k=1,2,\cdots,n-1$.
\end{lemma}

\begin{lemma}\label{vc}
Let $A$ be a C*-algebra. Let $a,b\in A^+$ and $\epsilon>0$. The following statements hold.

(i) If $a$ is Murray-von Neumann equivalent (and hence Cuntz equivalent) to $b$, then $(a-\epsilon)_+$ is Murray-von Neumann equivalent to $(b-\epsilon)_+$.

(ii) If $\|a\|,\|b\|\le 1$, $\epsilon\le 1$, and $a\in \overline{(b-\epsilon)_+A(b-\epsilon)_+}$, then $\|ba-a\|\le 2(1-\epsilon)$.

\end{lemma}
\begin{proof}
(i) We must show that $(x^*x-\epsilon)_+$ is Murray-von Neumann equivalent to $(x^*x-\epsilon)_+$.  Consider the polar decomposition $x = v|x|$ of $x$ in the bidual of A. The element $y = v(x^*x-\epsilon)_+^{1/2}$ belongs to $A$ and satisfies
$y^*y = (x^*x-\epsilon)_+$ and $yy^* = (xx^*-\epsilon)_+$.

(ii) If $\|b\|\le \epsilon$ then the statement is trivial. Let us suppose that $\|b\|>\epsilon>0$. With $b_\epsilon=f_{\epsilon}(b)$, where $f_\epsilon(t)=\min(\|b\|,(t\|b\|)/\epsilon)$, we have $b_\epsilon(b-\epsilon)_+=\|b\|(b-\epsilon)_+$ and $\|b-b_\epsilon\|=\|b\|-\epsilon$. By the first equation, $b_\epsilon a=\|b\|a$, and hence by the second equation,
\[
\|ba-a\|\le \|ba-\|b\|a\|+\|(1-\|b\|)a\|\le 2(1-\epsilon).
\]
\end{proof}

\begin{theorem}\label{a-l}
Let $A$ be a C*-algebra and let $(X,v)$ be a rooted tree. Consider a Cuntz semigroup morphism $\alpha\colon \Cu(\mathrm{C}_0(X\setminus v))\to \Cu(A)$ satisfying  $\alpha[s_X]\le [s_A]$, where $s_X$ and $s_A$ are a strictly positive element of $\mathrm{C}_0(X\setminus v)$ and a positive element of $A$, respectively. Then for every $\epsilon>0$ there exists a $\ast$-homomorphism $\phi\colon \mathrm{C}_0(X\setminus v)\to A$ such that $d_W^{(X,v)}(\alpha, \Cu (\phi))<\epsilon$.
\end{theorem}
\begin{proof}
Let $N$ be a positive integer (to be specified later). Consider the set $\mathrm{G}(X,v)$ of generators of the C*-algebra $\mathrm{C}_0(X\setminus v)$ defined in Subsection \ref{g-r}. Recall that the elements of $\mathrm{G}(X,v)$ are indexed by the edges of $(X,v)$. Let $e$ be a fixed but arbitrary such edge. By Lemma \ref{interpolation}, with $n=2^N$ and $x_k=\alpha[(g_e-k/n)]$ for $k=0,1,\cdots,n-1$, there exists a positive element $a_e\in A\otimes \K$ of norm at most 1 such that $\alpha[g_e]=[a_e]$, and
\begin{align}
[(a_e-(k+1)/2^N)_+]\ll \alpha[(g_e-k/2^N)_+]\ll [(a_e-k/2^N)_+],\label{int1}
\end{align}
for $k=1,2,\cdots,2^N-1$. Note that for each edge $e$,
\begin{align*}
\sum_{e' \text{ next to } e} \alpha[g_{e'}]\le \alpha[(g_e-t)_+]
\end{align*}
for all $0\le t<1$ (as this holds before applying $\alpha$), and therefore
\begin{align}\label{int2}
\sum_{e' \text{ next to } e} \alpha[g_{e'}]\ll [(a_e-(2^N-1)/2^N)_+].
\end{align}

Using the stability of $A\otimes \K$, we may choose the elements $(a_e)_{e\in E(X,v)}$ in such a way that $a_{e}a_{e'}=0$ if $e$ and $e'$ are beside each other.

Let $0<\delta<1/2^N$ be such that \eqref{int1} and \eqref{int2} still hold when the elements $a_e$ are replaced with $(a_e-\delta)_+$ (the existence of the number $\delta$ follows from \eqref{int1} and \eqref{int2}, the definition of the relation $\ll$, and from the fact that for $b=(a_e-k/2^N)$, $k=0,1,\cdots, 2^N-1$ (in fact for any positive element $b$!), $[b]=\sup_\delta[(b-\delta)_+]$, and $(b-\delta')_+\ll (b-\delta'')_+$ if $\delta'>\delta''$). Let us construct as follows a family of positive elements $(b_e)_{e\in \E(X,v)}$ in $A$ such that $b_e$ is Murray-von Neumann equivalent to $(a_e-\delta)_+$, such that $b_{e}b_{e'}=0$ if $e$ and $e'$ are beside each other, and such that
\begin{align}\label{identity}
b_{e'}\in \overline{(b_e-(2^N-1)/2^N)_+A(b_e-(2^N-1)/2^N)_+}
\end{align}
if $b_{e'}$ is next to $b_e$. Note that by (ii) of Lemma \ref{vc} this last relation implies that
\[
\|b_eb_{e'}-b_{e'}\|<1/2^{N-1}
\]
for all edges $e$ and $e'$ such that $e'$ is next to $e$.

Let us carry out this construction inductively. Let us start by constructing the positive elements $b_e$ associated to the edges $e$ that have $v$ as a vertex. Denote by $e_i$, $i=1,2,\cdots,k$, the edges of $(X,v)$ with one vertex $v$. Using that the elements $a_{e_i}$, $i=1,2,\cdots,k$, are pairwise orthogonal we have
\[
\left[\sum_{i=1}^ka_{e_i}\right]=\sum_{i=1}^k[a_{e_i}]=\sum_{i=1}^k\alpha[g_{e_i}]=\alpha\left[\sum_{i=1}^kg_i\right]\le \alpha[s_X]\le [s_A].
\]
More briefly, $\sum_{i=1}^ka_{e_i}\preccurlyeq s_A$. By Lemma 2.2 of \cite{k-r}, there is $x\in A$ such that
\[
\sum_{i=1}^k(a_{e_i}-\delta)_+=\left(\sum_{i=1}^ka_{e_i}-\delta\right)_+=x^*x, \quad xx^*\in \overline{s_AAs_A}=A.
\]
Let $x=V|x|$ be the polar decomposition of the element $x$ in the bidual of $A$. Set $V(a_{e_i}-\delta)_+V^*=b_{e_i}$, for $i=1,2,\cdots,k$. Then the elements $b_{e_i}$, $i=1,2,\cdots,k$, belong to $A$ and are pairwise orthogonal.

Now let us suppose that we have constructed the positive element $b_{e}$ associated to the edge $e$ and let us construct the elements $b_{e'}$ associated to the edges $e'$ that are next to the edge $e$. By the choice of $\delta$, and the fact that $\alpha[g_e]=[a_e]$ for every edge $e\in \E(X,v)$, we obtain from relation \eqref{int2} that
\[
\sum_{e' \text{ next to } e} [a_{e'}]\ll [(a_e-\delta-(2^N-1)/2^N)_+].
\]
By (i) of Lemma \ref{vc} applied to the elements $(a_e-\delta)_+$ and $b_e$, and $\epsilon=(2^N-1)/2^N$, we have
\[
[(a_e-\delta-(2^N-1)/2^N)_+]=[(b_e-(2^N-1)/2^N)_+].
\]
Therefore,
\[
\sum_{e'\text{ next to }e} [a_{e'}]\le [(b_{e}-(2^N-1)/2^N)_+].
\]
Hence, since the terms of $(a_{e'})_{e'\text{ next to } e}$ are pairwise orthogonal,
\[
\sum_{e'\text{ next to }e} a_{e'}\preccurlyeq (b_{e}-(2^N-1)/2^N)_+.
\]
Therefore, by Lemma 2.2 of \cite{k-r} there is $y\in A\otimes \K$ such that
$$\sum_{e'\text{ next to } e} (a_{e'}-\delta)_+=y^*y,$$
and
$$yy^*\in \overline{(b_e-(2^N-1)/2^N)_+A(b_e-(2^N-1)/2^N)_+}.$$

Let $y=W|y|$ be the polar decomposition of $y$ in the bidual of $A\otimes \K$. Set $W(a_{e'}-\delta)_+W^*=b_{e'}$. It is clear that the positive elements $b_{e'}$, $e'$ next to $e$, satisfy the required conditions.

Following this procedure we construct positive contractions $(b_e)_{e\in \mathrm{E}(X,v)}$---one positive element $b_e$ for each generator $g_e$ of $\mathrm{C}_0(X\setminus v)$---such that $b_eb_{e'}=0$ if $e$ and $e'$ are beside each other, and $\|b_eb_{e'}-b_{e'}\|<1/2^{N-1}$ if $e'$ is next to $e$. It follows that these elements form a $1/2^{N-1}$ representation in $A$ of the relations \eqref{relations}.

Fix $\epsilon>0$. Choose $n\geq 1$ such that $1/2^{n-1}<\epsilon$. Using the weak stability of the relations \eqref{relations}---which follows from Theorem 5.1 of \cite{Loring1} and Theorem 14.1.4 of \cite{Loring}---we can choose $N>n+1$ such that there are positive contractions $c_e\in A$ satisfying $c_ec_{e'}=0$ if $e$ and $e'$ are beside each other, $c_ec_{e'}=c_{e'}$ if $e'$ is next to $e$, and
\begin{align}\label{inequality}
\|b_e-c_e\|<1/2^{n+1}.
\end{align}
The elements $c_e$, $e\in \mathrm{E}(X,v)$, form a representation of the relations \eqref{relations} in the C*-algebra $A$. Therefore, they induce a $\ast$-homomorphism $\phi\colon \mathrm{C}_0(X\setminus v)\to A$, which is defined on the generators of $\mathrm{C}_0(X\setminus v)$ by $\phi(g_e)=c_e$. Let us prove that $\phi$ is the desired homomorphism.

Fix an edge $e$ of $(X,v)$. By Corollary 9.1 of \cite{c-e} applied to the elements $b_e-k/2^{n+1}$ and $c_e-k/2^{n+1}$, $k=0,1,\cdots,2^{n+1}$, (or by Lemma 1 of \cite{l-l} applied to the elements $b_e$ and $c_e$) we have
\begin{equation}\label{be-ce}
\begin{aligned}
&[(b_e-(k+1)/2^{n+1})_+]\le [(c_e-k/2^{n+1})_+],\\
&[(c_e-(k+1)/2^{n+1})_+]\le [(b_e-k/2^{n+1})_+],
\end{aligned}
\end{equation}
for $k=0,1,\cdots,2^{n+1}-1$. Since $b_e$ is Murray-von Neumann equivalent to $(a_e-\epsilon)_+$, by (i) of Lemma \ref{vc} the relation \eqref{int1} holds for $k=1,2,\cdots, 2^N-1$ when $a_e$ is replaced with $b_e$. Therefore,
\begin{align*}
[(b_e-(k+1)/2^N)_+]\le\alpha[(g_e-k/2^N)_+]\le [(b_e-k/2^N)_+],
\end{align*}
for $k=1,2,\cdots, 2^N-1$. Since $N>n+1$ the preceding equation implies that
\begin{align}\label{be-alpha}
[(b_e-(k+1)/2^{n+1})_+]\le\alpha[(g_e-k/2^{n+1})_+]\le [(b_e-k/2^{n+1})_+],
\end{align}
for $k=1,2,\cdots, 2^{n+1}-1$.

It follows from the inequalities \eqref{be-ce} and \eqref{be-alpha} that
\begin{align*}
[(c_e-(k+1)/2^n)_+]&=[(c_e-(2k+2)/2^{n+1})_+]\le [(b_e-(2k+1)/2^{n+1})_+]\\
&\le \alpha[(g_e-2k/2^{n+1})_+]=\alpha[(g_e-k/2^n)_+],
\end{align*}
and
\begin{align*}
\alpha[(g_e-(k+1)/2^n)_+]&=[(b_e-(k+1)/2^n)_+]=[(b_e-(2k+2)/2^{n+1})_+]\\
&\le [(c_e-2k/2^{n+1})_+]=[(c_e-k/2^n)_+],
\end{align*}
for $k=1,2,\cdots,2^n-1$.
These inequalities may be rewritten as
\begin{align*}
& \mathrm{Cu}(\phi)[(g_e-(k+1)/2^n)_+)]\le \alpha[(g_e-k/2^n)_+],\\
&\alpha[(g_e-(k+1)/2^n)_+]\le \mathrm{Cu}(\phi)[(g_e-k/2^n)_+],
\end{align*}
for $k=1,2,\cdots,2^n-1$.

Any interval of length $1/2^{n-1}$ contains an interval of the form $(k/2^n, (k+1)/2^n)$ for some integer $k$. Thus, for every $t\in [0,1]$ there exists $k$ such that $(k/2^n, (k+1)/2^n)\subseteq (t,t+1/2^{n-1})$. It follows from the preceding inequalities that
\begin{align*}
\Cu(\phi)[(g_e-t-1/2^{n-1})_+]&\le \Cu(\phi)[(g_e-(k+1)/2^n)_+]\\
&\le \alpha[(g_e-k/2^n)_+]\\
&\le \alpha[(g_e-t)_+],
\end{align*}
for $t\in [0,1]$. Interchanging the roles of $\Cu(\phi)$ and $\alpha$ (noting that they are symmetric) we also have
\begin{align*}
\alpha[(g_e-t-1/2^{n-1})_+]\le \Cu(\phi)[(g_e-t)_+],
\end{align*}
for $t\in [0,1]$. These inequalities can be restated as
\[
d_W(\Cu(\phi)\circ \Cu(\chi_{g_e})), \alpha\circ \Cu(\chi_{g_e}))\le 1/2^{n-1}.
\]
(Note that in the definition of the pseudometric $d_W$ it is enough to take $t$ in $[0,1]$.) Since the inequality above holds for all $e\in \E(X,v)$ we conclude that
\[
d_W^{(X,v)}(\Cu(\phi),\alpha)\le 1/2^{n-1}<\epsilon.
\]
\end{proof}

\section{Proof of Theorem \ref{homomorphism} and Corollary \ref{classification}}
\begin{lemma}\label{un_eq}
Let $A$ and $B$ be C*-algebras with $B$ of stable rank one. Let $\phi,\psi\colon A\to B$ be $\ast$-homomorphisms such that $\phi$ is approximately unitarily equivalent to $\psi$ with the unitaries taken in the unitization of $B\otimes \K$. Then $\phi$ is approximately unitarily equivalent to $\psi$ with the unitaries taken in the unitization of $B$.
\end{lemma}
\begin{proof}
Let $F$ be a finite subset of $A$, and let $0<\epsilon<1$. Let us choose a positive element $a\in A$ such that $\|a\|<1$, and
\begin{align}\label{a_eq}
\|a^2f-f\|<\epsilon, \quad \|afa-f\|<\epsilon,
\end{align}
for all $f\in F$ ($a$ can be chosen to be a suitable element of an approximate unit of the C*-algebra generated by the elements of $F$). Since $\phi$ and $\psi$ are approximately unitarily equivalent with unitaries taken in $(B\otimes \K)\,\widetilde{}$, there exists a unitary $U\in (B\otimes \K)\,\widetilde{}$ such that
\begin{align}\label{af_eq}
\|U\phi(a^2)U^*-\psi(a^2)\|<\epsilon,\quad\|U\phi(f)U^*-\psi(f)\|<\epsilon,
\end{align}
for all $f\in F$. Set $\psi(a)U\phi(a)=z$. Then by the triangle inequality and the second inequality in \eqref{a_eq} and \eqref{af_eq} we have
\begin{align*}
&\|z\phi(f)z^*-\psi(f)\|=\|\psi(a)U\phi(afa)U^*\psi(a)-\psi(f)\|\le\\
&\le \|\psi(a)U(\phi(afa)-\phi(f))U^*\psi(a)\|+\|\psi(a)(U\phi(f)U^*-\psi(f))\psi(a)\|+\|\psi(afa)-\psi(f)\|\le\\
&\le \epsilon+\epsilon+\epsilon=3\epsilon,
\end{align*}
for all $f\in F$.
Also, using the first inequality in \eqref{af_eq} we obtain that
\begin{align*}
\|z^*z-\phi(a^4)\|=\|\phi(a)(U\psi(a^2)U^*-\phi(a^2))\phi(a)\|<\epsilon.
\end{align*}
It follows that $\|z\|<1$ and
\begin{align}\label{za_eq}
\|z\phi(f)z^*-\psi(f)\|<3\epsilon, \quad \||z|-\phi(a^2)\|<\sqrt{\epsilon},
\end{align}
for all $f\in F$, with Lemma \ref{ineq} being used to obtain the last inequality.
Since the C*-algebra $B$ has stable rank one we may assume that $z$ is an invertible element of $\widetilde B$.
Let us denote by $W$ the unitary in the polar decomposition of the element $z$. Then by the triangle inequality and the inequalities \eqref{a_eq} and \eqref{za_eq} we have
\begin{align*}
&\| W\phi(f)W^*-\psi(f)\|\le\|W(\phi(f)-\phi(a^2f))W^*\|+\|W\phi(a^2)\phi(f)W^*-z\phi(f)W^*\|+\\
&+\|z\phi(f)W^*-z\phi(fa^2)W^*\|+\|z\phi(f)\phi(a^2)W^*-z\phi(f)z^*\|+\|z\phi(f)z^*-\psi(f)\|<\\
& <\epsilon+\sqrt{\epsilon}+\epsilon+\sqrt{\epsilon}+\epsilon\le 5\sqrt{\epsilon},
\end{align*}
for all $f\in F$.

Since the finite subset $F$ and the positive number $\epsilon$ are arbitrary we conclude that $\phi$ and $\psi$ are approximately unitarily equivalent with unitaries taken in the unitization of $B$.
\end{proof}

Let $A$ and $B$ be C*-algebras such that $A$ has a strictly positive element---say $s_A$. Let us say that the ordered pair $(A,B)$ has the property (P) if for any Cuntz semigroup morphism $\alpha\colon \Cu(A)\to \Cu(B)$ such that $\alpha[s_A]\le [s_B]$, where $s_B$ is a positive element of $B$, there exists a $\ast$-homomorphism $\phi\colon A\to B$---unique up to approximate unitary equivalence---such that $\alpha=\Cu(\phi)$.

\begin{proposition}\label{prop(P)}
The following statements hold true:

(i) Let $B$ be a C*-algebra of stable rank one and let $(X,v)$ be a rooted tree. If the pair $(\mathrm{C}_0(X\setminus v),pBp)$ has the property (P) for every projection $p$ of $B$, then the pair $(\mathrm{C}(X),B)$ has the property (P).

(ii) If the pair of C*-algebras $(A,B)$ has the property (P) and $B$ has stable rank one, then the pair $(\M_n(A), B)$ has the property (P), for every $n\in \N$.

(iii) Let $C$ be a C*-algebra of stable rank one. If the pairs of C*-algebras $(A,D)$ and $(B,D)$ have the property (P) for all hereditary subalgebras $D$ of $C$, then the pair $(A\oplus B, C)$ has the property (P).

(iv) If the pairs of C*-algebras $(A_i, B)$ have the property (P) for a sequence
\begin{align*}
\xymatrix{
A_1\stackrel{\rho_1}\longrightarrow A_2\stackrel{\rho_2}\longrightarrow \cdots,
}
\end{align*}
then the pair $(\varinjlim (A_i,\rho_i), B)$ has the property (P).

(v) Let $A$, $B$, and $C$ be C*-algebras such that $A$ is stably isomorphic to $B$, and $C$ has stable rank one. If the pair $(A, C\otimes \K)$ has the property (P), then the pair $(B, C)$ has the property (P).
\end{proposition}
\begin{proof}
(i) Let $\alpha\colon \Cu(\mathrm{C}(X)) \to \Cu(B)$ be a morphism in the Cuntz category satisfying $\alpha([1_X])\le [s_B]$. Let us show that it is induced by a $\ast$-homomorphism $\varphi\colon \mathrm{C}(X)\to B$. As $B$ has stable rank one, the element $\alpha([1_X])$ appears as the Cuntz class of a projection of $B$, say $p$. We have the following identifications:
\begin{align*}
\Cu(pBp)&\cong\{x\in \Cu(B)\mid x\le \infty [p]\},\\
\Cu(\mathrm{C}_0(X\setminus v))&\cong \{x\in \Cu(\mathrm{C}(X))\mid x\le \infty [s_X]\},
\end{align*}
where $s_X$ is a strictly positive element of $\mathrm{C}(X, v)$. It follows that $\alpha[s_X]\le [p]$, and
\[
\alpha(\Cu(\mathrm{C}_0(X\setminus v)))\subseteq \Cu(pBp).
\]
By assumption the pair $(\Cu(\mathrm{C}_0(X\setminus v)), pBp)$ has the property (P). Therefore, there is a $\ast$-homomorphism $\phi\colon \mathrm{C}_0(X\setminus v)\to pBp$ such that $\Cu(\phi)$ is equal to the restriction of $\alpha$ to $\Cu(\mathrm{C}_0(X\setminus v))$.
Let us consider the restriction of $\alpha$ to $\Cu(\mathrm{C}_0(X\setminus v))$. Denote by $\tilde\phi\colon \mathrm{C}(X)\to B$ the $\ast$-homomorphism which agrees with $\phi$ on $\mathrm{C}_0(X\setminus v)$ and satisfies $\tilde\phi(1)=p$. Then $\Cu(\tilde\phi)=\alpha$. (The proof of this fact is similar to the proof of (i) of Proposition \ref{metrics}.)

Now let us show that if $\phi, \psi\colon \mathrm{C}(X)\to B$ are $\ast$-homomorphisms such that $\Cu(\phi)=\Cu(\psi)$, then they are approximately unitarily equivalent. Since $\Cu(\phi)[1_X]=\Cu(\psi)[1_X]$  and $B$ has stable rank one we may assume that $\phi(1_X)=\psi(1_X)=p$, where $p\in B$ is a projection. Let us denote by $\phi', \psi'\colon \mathrm{C}_0(X\setminus v)\to  pBp\subseteq B$ the restrictions of the $\ast$-homomorphisms $\phi$ and $\psi$ to $\mathrm{C}_0(X\setminus v)$. Since $\Cu(\phi)=\Cu(\psi)$, we have $d_W^{(X,v)}(\Cu(\phi'), \Cu(\psi))=0$. Using the relation between $d_U^{(X,v)}$ and $d_W^{(X,v)}$ established in Theorem \ref{pseu-rel}, we conclude that $d_{U}^{(X,v)}(\phi', \psi')=0$ (the unitaries taken in $pBp$). It follows now that $\phi$ and $\psi$ are approximately unitarily equivalent, as desired.

(ii) Let us start by showing that any Cuntz semigroup morphism $\alpha\colon \Cu(\M_n(A))\to B$, satisfying $\alpha[s_A\otimes 1_n]\le [s_B]$, where $1_n$ denotes the unit of $\M_n(\C)$, can be lifted to a  $\ast$-homomorphism $\phi\colon \M_n(A)\to B$ such that $\Cu(\phi)=\alpha$.

Let us consider the inclusion map $i_A\colon A\to \M_n(A)$ given by $i_A(a)=a\otimes e_{1,1}$. The map $\Cu(i_A)$ is an isomorphism. We have
\begin{align*}
\xymatrix{
\Cu(A)\ar[r]^(0.43){\Cu(i_A)}&\Cu(\M_n(A))\ar[r]^(0.55){\alpha}&\Cu(B).
}
\end{align*}
Since the pair $(A,B)$ has the property (P), there exists a $\ast$-homomorphism $\psi\colon A\to B$ such that $\alpha \circ \Cu(i_A)=\Cu(\psi)$. It follows that $\alpha=\Cu(\psi)\circ \Cu(i_A)^{-1}$.

Using the commutativity of the diagram
\[
\xymatrix{
        A\ar[r]^{\psi}\ar[d]^{i_A} & B\ar[d]^{i_B}\\
        \M_n(A)\ar[r]^{\psi \otimes \id}& \M_n(B)
        }
\]
we obtain that $\Cu(\psi \otimes \id)=\Cu(i_B)\circ \Cu(\psi)\circ \Cu(i_A)^{-1}$, and hence $\Cu(\psi \otimes
\id)=\Cu(i_B)\circ \alpha$. Therefore,
\begin{align*}
[(\psi \otimes \id)(s_A\otimes 1_n)]=\Cu(i_B)\circ \alpha [s_A\otimes 1_n]\le \Cu(i_B)[s_B]=[s_B\otimes e_{1,1}].
\end{align*}
Since the C*-algebra $A$ has stable rank one, $\M_n(A)$ has stable rank one. It follows by (ii) of Proposition \ref{prelimprop} that there is an element $x\in \M_n(B)$ such that
\begin{align*}
(\psi \otimes \id)(s_A\otimes 1_n)=x^*x,\quad xx^*\in \overline{(s_B\otimes e_{1,1})\M_n(B)(s_B\otimes e_{1,1})}.
\end{align*}
Let $x=V|x|$ be the polar decomposition of the element $x$ in the bidual of $\M_n(B)$. Then $\Ad(V^*)\circ (\psi \otimes \id)$ is a $\ast$-homomorphism with image contained in the hereditary subalgebra $\overline{(s_B\otimes e_{1,1})\M_n(B)(s_B\otimes e_{1,1})}$. Also, $\Cu(\Ad(V^*)\circ (\psi \otimes \id))=\Cu(\psi \otimes \id)$. Denote by
\[
\gamma\colon \overline{(s_B\otimes e_{1,1})\M_n(B)(s_B\otimes e_{1,1})}\to B
\]
the isomorphism defined by $\gamma(b\otimes e_{1,1})=b$ for all $b\in B$. Then $\Cu(\gamma)=\Cu(i_B)^{-1}$, since $\Cu(B\otimes{e_{1,1}})=\Cu(B\otimes \K)$. Set $\gamma\circ \Ad(V^*)\circ (\psi\otimes \id)=\phi$. It follows that $\Cu(\phi)=\alpha$.

Let $\phi,\psi\colon \M_n(A)\to B$ be $\ast$-homomorphisms such that $\Cu(\phi)=\Cu(\psi)$. Let us show that $\phi$ and $\psi$ are approximately unitarily equivalent. Let us consider a finite subset of $\M_n(A)$ of the form
\begin{align*}
\{a\otimes e_{k,l}\mid 1\le k,l\le n, a\in F, \|a\|<1\},
\end{align*}
where $F$ is a finite subset of $A$. Let $0<\epsilon<1$ and let $b\in A$ be a positive contraction such that
\begin{align}\label{b}
\|bab-a\|<\epsilon,\quad \|b^2a-a\|<\epsilon,
\end{align}
for all $a\in F$. Since the pair $(A, B)$ has the property (P) and $\Cu(\phi\circ i_A)=\Cu(\phi\circ i_A)$, where $i_A\colon A\to \M_n(A)$ is the inclusion map, there is a unitary $U\in \widetilde B$ such that
\begin{align}\label{Ue11}
\|U\phi(a\otimes e_{1,1})U^*-\psi(a\otimes e_{1,1})\|<\epsilon,
\end{align}
for all $a\in F\cup \{b^2\}$. Let us set
\begin{align*}
\sum_{i=1}^n\psi(b\otimes e_{i,1})U\phi(b\otimes e_{1,i})=z.
\end{align*}
Using the inequalities \eqref{b} and \eqref{Ue11} we get for all $a\in F$ that
\begin{align*}
&\|z\phi(a\otimes e_{k,l})z^*-\psi(a\otimes e_{k,l})\|=\|\psi(b\otimes e_{k,1})U\phi(bab\otimes e_{1,1})U^*\psi(b\otimes e_{1,l})-\psi(a\otimes e_{k,l})\|\le\\
&\le\|\psi(b\otimes e_{k,1})(U\phi(bab\otimes e_{1,1})U^*-\psi(a\otimes e_{1,1}))\psi(b\otimes e_{1,l})\|+\|\psi(bab\otimes e_{k,l})-\psi(a\otimes e_{k,l})\|\\
&<2\epsilon +\epsilon=3\epsilon.
\end{align*}

Also, using the inequality \eqref{Ue11} we have
\begin{align*}
\|z^*z-\phi(b^4\otimes 1_n)\|=\|\sum_{i=1}^n\phi(b\otimes e_{i,1})(U^*\psi(b^2\otimes e_{1,1})U-\phi(b^2\otimes e_{1,1}))\phi(b\otimes e_{1,i})\|<\epsilon.
\end{align*}
Hence, we have found an element $z\in B$ such that
\begin{align}\label{z}
&\||z|-\phi(b^2\otimes 1_n)\|<\sqrt{\epsilon},\\
&\|z\phi(a\otimes e_{k,l})z^*-\psi(a\otimes e_{k,l})\|< 3\epsilon,\label{ze}
\end{align}
for all $a\in F$ and $k,l=1,2,\cdots, n$.
Since $B$ has stable rank one we can approximate the element $z$ by an invertible element $z'\in \widetilde B$ in such a way that inequalities \eqref{z} and \eqref{ze} still hold when $z$ is replaced by $z'$. Denote by $W$ the unitary in the polar decomposition of the invertible element $z'$. Then,
\begin{align*}
&\|W\phi(a\otimes e_{k,l})W^*-\psi(a\otimes e_{k,l})\|=\|W(\phi(a\otimes e_{k,l})-\phi(b^2a\otimes e_{k,l}))W^*\|+\\
&+\|W(\phi(b^2\otimes e_{k,l})-|z'|)\phi(a\otimes e_{k,l})W^*\|+\| W|z'|(\phi(a\otimes e_{k,l})-\phi(ab^2\otimes e_{k,l}))W^*\|+\\
&+\| W|z'|\phi(a\otimes e_{k,l})(\phi(b^2\otimes e_{k,l})-|z'|)W^*\|+\| z'\phi(a\otimes e_{k,l})(z')^*-\psi(a\otimes e_{k,l})\|\\
&< \epsilon+\sqrt{\epsilon}+2\epsilon+2\sqrt{\epsilon}+3\epsilon<9\sqrt{\epsilon},
\end{align*}
for all $a\in F$ and $k,l=1,2,\cdots,n$. Therefore, the unitary $W\in \widetilde B$ satisfies
\[
\|W\phi(a\otimes e_{k,l})W^*-\psi(a\otimes e_{k,l})\|<9\sqrt{\epsilon},
\]
for all $a\in F$ and $k,l=1,2,\cdots,n$. This proves that the $\ast$-homomorphisms $\phi$ and $\psi$ are approximately unitarily equivalent.

(iii) Let $\alpha\colon \Cu(A\oplus B)\to \Cu(C)$ be such that $\alpha[s_A\oplus s_B]\le [s_C]$, where $s_A$, $s_B$ are strictly positive elements of $A$ and $B$, and $s_C$ is a positive element of $C$. Let us consider the positive elements $a,b\in C\otimes \K$ such that $\alpha[s_A\oplus 0]=[a]$ and $\alpha[0\oplus s_B]=[b]$. By the stability of $C\otimes\K$ we may assume that $a$ and $b$ are orthogonal. We have
\[
\alpha[s_A\oplus s_B]=[a]+[b]=[a+b]\le [s_C].
\]
Applying (ii) of Proposition \ref{prelimprop} to the C*-algebra $C$ and the positive elements $a+b$ and $s_C$, we can find an element $x\in C\otimes \K$ such that
\[
a+b=x^*x,\quad xx^*\in C.
\]
Consider the polar decomposition $x=V|x|$ of $x$ in the bidual of $C\otimes \K$. Set $VaV^*=a'$ and $VbV^*=b'$. Then $a'$ and $b'$ are orthogonal elements of  $C^+$ satisfying $[a]=[a']$ and $[b]=[b']$. Also, we have the following natural isomorphisms:
\begin{align*}
\Cu(\overline{a'Aa'})\cong \{[z]\in \Cu(C)\mid [z]\le \infty[a']\},\\
\Cu(\overline{b'Ab'})\cong \{[z]\in \Cu(C)\mid [z]\le \infty[b']\}.
\end{align*}
Using these identifications and the fact that $\Cu(A\oplus B)$ is naturally isomorphic to $\Cu(A)\oplus \Cu(B)$, we can identify the morphism $\alpha$ with a pair of Cuntz semigroup morphisms $(\alpha_1,\alpha_2)$,
\begin{align*}
\alpha_1\colon \Cu(A)\to \Cu(\overline{a'Aa'}),\quad \alpha_1[s_A]\le [a_1],\\
\alpha_2\colon \Cu(B)\to \Cu(\overline{b'Ab'}),\quad \alpha_1[s_B]\le [b_2].
\end{align*}
Since by hypothesis the pairs $(A,\overline{a'Aa'})$, and $(B,\overline{b'Ab'})$ have the property (P) we can find $\ast$-homomorphisms $\phi_1\colon A\to \overline{a'Aa'}$, $\phi_2\colon A\to \overline{b'Ab'}$ such that $\Cu(\phi_1)=\alpha_1$ and $\Cu(\phi_2)=\alpha_2$. It follows that the $\ast$-homomorphism $\phi=(\phi_1, \phi_2)\colon A\oplus B\to C$ induces the morphism $\alpha$ at the level of Cuntz semigroups.

Let $\phi,\psi\colon A\oplus B\to C$ be $\ast$-homomorphisms such that $\Cu(\phi)=\Cu(\psi)$. Let us show that $\phi$ and $\psi$ are approximately unitarily equivalent.

Since $\Cu(\phi)=\Cu(\psi)$ we have $\phi(s_A)\sim \psi(s_A)$, and $\phi(s_B)\sim \psi(s_B)$ (we are using the identifications $A\oplus 0\cong A$ and $0\oplus B\cong B$). Again by (ii) of Proposition \ref{prelimprop} we can find elements $x_1,x_2\in C$ such that
\begin{align*}
&\phi(s_A)=x_1^*x_1,\quad x_1x_1^*\in \overline{\psi(s_A)C\psi(s_A)},\\
&\phi(s_B)=x_2^*x_2,\quad x_2x_2^*\in \overline{\psi(s_B)C\psi(s_B)}.
\end{align*}
Set $x_1+x_2=x$. Consider the partial isometry $V$ in the polar decomposition of $x$ in the bidual of $C$. Since the elements $x_1$ and $x_2$ satisfy the orthogonality relations $x_1^*x_2=x_2x_1^*=0$, we have
\begin{align*}
x_1=V|x_1|,\quad x_2=V|x_2|.
\end{align*}
It follows from these identities that the map $\Ad(V^*)\circ \phi$ is a $\ast$-homomorphism. Also, we have $\Cu(\Ad(V^*)\circ \phi)=\Cu(\phi)$. Let us denote by $\phi_A$ and $\phi_B$ the restrictions of the $\ast$-homomorphism $\Ad(V^*)\circ \phi$ to the C*-algebras $A$ and $B$, respectively. Then
\[
\phi_A(A)\subseteq \overline{\psi(s_A)C\psi(s_A)},\quad \phi_B(B)\subseteq \overline{\psi(s_B)C\psi(s_B)}.
\]
Since by hypothesis the pair of C*-algebras $(A, \overline{\psi(s_A)C\psi(s_A)})$ has the property (P), the $\ast$-homomorphism $\phi_A$ is approximately unitarily equivalent---with the unitaries taken in the unitization of the C*-algebra $\overline{\psi(s_A)C\psi(s_A)}$---to the restriction of the *-homomorphism $\psi$ to the C*-algebra $A$. In similar way the $\ast$-homomorphism $\phi_B$ and the restriction of the $\ast$-homomorphism $\psi$ to the C*-algebra $B$ are approximately unitarily equivalent in the unitization of the C*-algebra $\overline{\psi(s_B)C\psi(s_B)}$. It follows that the $\ast$-homomorphisms $\Ad(V^*)\circ \phi$ and $\psi$ are approximately unitarily equivalent in the unitization of the C*-algebra $C$.
In order to complete the proof let us show that the $\ast$-homomorphisms $\Ad(V^*)\circ \phi$ and $\phi$ are approximately unitarily equivalent. Recall that $V$ is the partial isometry in the polar decomposition of the element $x=x_1+x_2$. It follows by Proposition \ref{prelimprop} applied to the C*-algebra $C$ and the element $x$ that for every $\epsilon>0$ and every finite subset $F$ of the hereditary algebra $\overline{x^*Cx}$ there is a unitary $U\in C^\sim$ such that
\[
\|VzV^*-UzU^*\|<\epsilon,
\]
for all $z\in F$. This implies that $\Ad(V^*)\circ \phi$ and $\phi$ are approximately unitarily equivalent, since $\phi(A\oplus B)\subseteq \overline{x^*Cx}$.

(iv) Let $A=\varinjlim (A_i,\rho_i)$. For each $1\le i<j$ let $\rho_{i,j}\colon A_i\to A_j$ denote the $\ast$-homomorphism $\rho_{j-1}\circ\rho_{j-2}\circ\cdots\circ\rho_i$. Also, for each $1\le i$ let $\rho_{i,\infty}\colon A_i\to A$ denote the $\ast$-homomorphism given by the inductive limit.

Let $\alpha\colon A\to B$ be a Cuntz semigroup morphism such that $\alpha[s_A]\le [s_B]$, where $s_A$ is a strictly positive elements of $A$, and $s_B$ is a positive element of $B$. For each $i\ge 1$ set $\alpha\circ \Cu(\rho_{i,\infty})=\alpha_i$. We have
\[
\alpha_i[s_{A_i}]=\alpha[\rho_{i,\infty}(s_{A_i})]\le \alpha[s_A]\le [s_B].
\]
where $s_{A_i}$ denotes a strictly positive element of $A_i$. Since the pairs $(A_i, B)$ have the property (P) for all $i\ge 1$, there exist $\ast$-homomorphisms $\phi_i\colon A_i\to B$, such that $\Cu(\phi_i)=\alpha_i$. For each $i$ we have $\Cu(\phi_i)=\Cu(\phi_{i+1}\circ \rho_i)$. Hence the $\ast$-homomorphisms $\phi_i$ and $\phi_{i+1}\circ \rho_i$ are approximately unitarily equivalent. By Subsection 2.3 of \cite{elliott-AT} there exists a $\ast$-homomorphism $\phi\colon A\to B$ such that for every $i\ge 1$ the $\ast$-homomorphisms $\phi\circ \rho_{i,\infty}$ and $\phi_i$ are approximately unitarily equivalent. Since the Cuntz functor is equal in $\ast$-homomorphisms that are approximately unitarily equivalent we have
$\Cu(\phi\circ \rho_{i,\infty})=\Cu(\phi_i)$. Therefore, $\Cu(\phi)\circ \Cu(\rho_{i,\infty})=\alpha_i$ for all $i\ge 1$. By the universal property of the inductive limit this implies that $\alpha=\Cu(\phi)$.

To conclude the proof of (iv) let us show that if homomorphisms $\phi,\psi\colon A\to B$ are such that $\Cu(\phi)=\Cu(\psi)$, then they are approximately unitarily equivalent. For each $i\ge 1$ set $\phi\circ \rho_{i,\infty}=\phi_i$, and $\psi\circ \rho_{i,\infty}=\psi_i$. We have $\Cu(\phi_i)=\Cu(\psi_i)$ for $i\ge 1$. Since for each $i\ge 1$ the pair $(A_i, B)$ has the property (P) the $\ast$-homomorphisms $\phi_i$ and $\psi_i$ are approximately unitarily equivalent. It follows that the $\ast$-homomorphisms $\phi$ and $\psi$ are approximately unitarily equivalent.

(v) The pair $(A, C\otimes \K)$ has the property (P) by assumption, and it follows by Lemma \ref{un_eq} that the pair $(A,C)$ also has the property (P). It follows by (ii) and (iv) that also the pairs $(A\otimes \K, C)$ and $(A\otimes\K, C\otimes \K)$ have the property (P). Since $A\otimes \K\cong B\otimes \K$, the pairs $(B\otimes \K, C)$ and $(B\otimes\K, C\otimes \K)$ have the property (P).

Let $\alpha\colon \Cu(B)\to \Cu(C)$ be a Cuntz semigroup morphism such that $\alpha[s_B]\le [s_C]$, where $s_B$ is a strictly positive element of $B$, and $s_C$ is a positive element of $C$. Let $i_B\colon B\to B\otimes \K$ and $i_C\colon C\to C\otimes \K$ denote the inclusion maps $i_B(b)=b\otimes e_{1,1}$ and $i_C(c)=c\otimes e_{1,1}$. Then $\Cu(i_B)$ and $\Cu(i_C)$ are isomorphisms. Since the pair
$(B\otimes\K, C\otimes \K)$ has the property (P), there exists a $\ast$-homomorphism $\phi\colon B\otimes\K\to C\otimes\K$ such that $\Cu(\phi)=\Cu(i_C)\circ\alpha \circ\Cu(i_B)^{-1}$. We have
\[
\Cu(\phi\circ i_B)[s_B]=(\Cu(i_C)\circ\alpha)[s_B]\le [s_C\otimes e_{1,1}].
\]
By Proposition \ref{prelimprop} there exists $x\in C\otimes \K$ such that
\begin{align*}
(\phi\circ i_B)(s_B)=x^*x, \quad xx^*\in\overline{(s_C\otimes e_{1,1})(C\otimes \K)(s_C\otimes e_{1,1})}.
\end{align*}
Consider the polar decomposition $x=V|x|$ of $x$ in the bidual of $C\otimes \K$. Then
\[
(\Ad(V^*)\circ\phi\circ i_B)(B)\subseteq \overline{(s_C\otimes e_{1,1})(C\otimes \K)(s_C\otimes e_{1,1})}=C\otimes e_{1,1},
\]
and $\Cu(\Ad(V^*)\circ\phi\circ i_B)=\Cu(\phi\circ i_B)$. Denote by $\gamma\colon C\otimes e_{1,1}\to C$ the $\ast$-isomorphism $\gamma(c\otimes e_{1,1})=c$ for all $c\in C$. Then $\Cu(\gamma)=\Cu(i_C)^{-1}$, since $\Cu(C\otimes{e_{1,1}})=\Cu(C\otimes \K)$. Set $\gamma\circ\Ad(V^*)\circ\phi\circ i_B=\phi'$. It follows that $\phi'\colon B\to C$, and
\[
\Cu(\phi')=\Cu(\gamma\circ\Ad(V^*)\circ\phi\circ i_B)=\Cu(i_C)^{-1}\circ \Cu(\phi)\circ\Cu(i_B)=\alpha.
\]

Now let us show that if $\phi,\psi\colon B\to C$ are $\ast$-homomorphisms such that $\Cu(\phi)=\Cu(\psi)$, then they are approximately unitarily equivalent, with the unitaries taken in the unitization of $C$.

Since $\Cu(\phi)=\Cu(\psi)$ it follows that $\Cu(\phi\otimes \id)=\Cu(\psi\otimes \id)$, where $\id\colon \K\to \K$ is the identity map. This implies that $\phi\otimes \id$ is approximately unitarily equivalent to $\psi\otimes \id$, with the unitaries taken in $(C\otimes \K)\,\widetilde{}$. Hence, $(\psi\otimes \id)\circ i_B$ and $(\psi\otimes \id)\circ i_B$ are approximately unitarily equivalent. Since the images of both maps are contained in $C$, $\phi$ and $\psi$ are approximately unitarily equivalent with the unitaries taken in the unitization of $C$, by Lemma \ref{un_eq}.

\end{proof}

\begin{proof}[Proof of Theorem \ref{homomorphism}] By the statements (i), (iii), (iv), and (v) of Proposition \ref{prop(P)}, and by Proposition \ref{c-t} it is enough to prove the theorem in the case $A=\mathrm{C}_0(X\setminus v)$, where $(X,v)$ is a rooted tree.

The uniqueness of the homomorphism $\phi$ is a special case of Theorem \ref{m-rel}. Let us prove its existence.
By Theorem \ref{a-l}, for every $n$ there exists $\phi_n\colon \mathrm{C}_0(X\setminus v)\to B$ such that
$d_W^{(X,v)}(\mathrm{Cu}(\phi_n),\alpha)<1/2^{n+2}$. It follows by Theorem \ref{m-rel} that
\begin{align*}
d_U^{(X,v)}(\phi_n, \phi_{n+1})\le (2N+2)d_W^{(X,v)}(\Cu(\phi_n), \Cu(\phi_{n+1}))<\frac{1}{2^n}(2N+2).
\end{align*}
This implies that $(\phi_n)_n$ is a Cauchy sequence with respect to the pseudometric $d_U^{(X,v)}$. By (ii) of Proposition \ref{metrics},  the space $\Hom (\mathrm{C}_0(X\setminus v), B)$ is complete with respect to $d_U^{(X,v)}$. Hence, there exists $\phi\colon \mathrm{C}_0(X\setminus v)\to B$ such that $d_U^{(X,v)}(\phi, \phi_n)\to 0$. By the first inequality of Theorem \ref{pseu-rel},
\begin{align*}
d_W^{(X,v)}(\Cu(\phi),\alpha) &\le d_W^{(X,v)}(\Cu(\phi), \Cu(\phi_n))+d_W^{(X,v)}(\Cu(\phi_n), \alpha),\\
& \le d_U^{(X,v)}(\phi, \phi_n)+d_W^{(X,v)}(\Cu(\phi_n), \alpha)\to 0.
\end{align*}
In other words, $d_W^{(X,v)}(\Cu(\phi),\alpha)=0$. Since, as shown in (i) of Proposition \ref{metrics}, $d_W^{(X,v)}$ is a metric, we have $\Cu(\phi)=\alpha$, as desired.
\end{proof}

\begin{proof}[Proof of Corollary \ref{classification}]
Let $A$ and $B$ be sequential inductive limits of separable continuous-trace C*-algebras with spectrum homeomorphic to a disjoint union of trees and trees with a point removed. Let $\alpha\colon \Cu(A)\to \Cu(B)$ be a Cuntz semigroup isomorphism such that $\alpha[s_A]=[s_B]$. Then by Theorem \ref{homomorphism}, there are $\ast$-homomorphisms $\phi\colon A\to B$ and $\psi\colon B\to A$ such that $\Cu(\phi)=\alpha$, and $\Cu(\psi)=\alpha^{-1}$. We have $\Cu(\phi\circ\psi)=\Cu(\id_A)$ and $\Cu(\psi\circ\phi)=\Cu(\id_B)$, where $\id_A$ and $\id_B$ denote the identity endomorphisms of $A$ and $B$. Therefore by the uniqueness part of Theorem \ref{homomorphism}, $\phi\circ\psi$ and $\id_A$, and $\psi\circ\phi$ and $\id_B$ are approximately unitarily equivalent. Hence, by Theorem 3 of \cite{elliott-cla} (cf. Section 4.3 of \cite{elliott-cla}), there exists an isomorphism $\rho\colon A\to B$ approximately unitarily equivalent to $\phi$, and in particular by definition of $\Cu(B)$ such that $\Cu(\rho)=\Cu(\phi)=\alpha$. The $\ast$-homomorphism $\rho$ is unique---up to approximate unitary equivalence---by Theorem \ref{homomorphism}.
\end{proof}

\subsection{Remarks}
It follows from (v) of Proposition \ref{prop(P)} that Theorem \ref{homomorphism} still holds if $A$ is taken to be a full hereditary subalgebra of a C*-algebra that can be written as a sequential inductive limit of separable continuous-trace C*-algebras with spectrum homeomorphic to a disjoint union of trees and trees with a point removed. The same applies to Corollary \ref{classification} for the C*-algebras $A$ and $B$. We don't know if in general every such full hereditary subalgebra can be written as a sequential inductive limit of continuous-trace C*-algebras with spectrum homeomorphic to a disjoint union of trees and trees with a point removed. We have the following partial result:

\begin{proposition} Let $A$ be a full hereditary subalgebra of a C*-algebra that can be written as a sequential inductive limit of continuous-trace C*-algebras with spectrum homeomorphic to a compact tree. Then $A$ can be written as a sequential inductive limit of continuous-trace C*-algebras with spectrum homeomorphic to a compact tree.
\end{proposition}
\begin{proof}
Let $B=\varinjlim B_i$, where each $B_i$ is a continuous-trace C*-algebras with spectrum homeomorphic to a tree. Let $A$ be a full hereditary subalgebra of $B$. Since $B$ has stable rank one, by Corollary 4 of \cite{c-e-i} the C*-algebra $A$ can be written as a sequential inductive limit of hereditary subalgebras of the C*-algebras $B_i$, $i=1,2,\cdots$. Let us denote these hereditary subalgebras by $A_i$, $i=1,2,\cdots$. It follows that each C*-algebra $A_i$ is a continuous-trace C*-algebra. Therefore, we have
\[
\varinjlim (\mathrm{C}_0(\hat{A_i})\otimes \K)\cong\varinjlim (A_i\otimes \K)=A\otimes\K\cong B\otimes \K.
\]
Since $B\otimes \K=\varinjlim (\mathrm{C}(\hat{B_i})\otimes\K)$, the C*-algebra $B\otimes \K$ contains a nonzero projection. It follows that for $i$ large enough the C*-algebras $A_i\otimes \K$ also contain  nonzero projections. Therefore, $\hat{A_i}$ is compact for $i$ large enough. Since for each $i=1,2,\cdots$, the C*-algebra $\mathrm{C}(\hat{A_i})\otimes \K$ is a hereditary subalgebra of $\mathrm{C}(\hat{B_i})\otimes \K$, and $\hat{B_i}$ is a compact tree, then  for $i$ large enough the set $\hat{A_i}$ is a compact tree. This concludes the proof of the proposition.
\end{proof}

\thanks{\textbf{Acknowledgments}}
The research of the second author was supported by the Natural Sciences and Engineering
Research Council of Canada. The third author was partially supported by DGI MICIIN-FEDER
MTM2008-06201-C02-01, and by the Comissionat per Universitats i
Recerca de la Generalitat de Catalunya. 

Part of this research was carried out in the 2008--2009 academic year while the first and third authors were postdoctoral fellows at the Fields Institute for Research in Mathematical Sciences with the support of Dr. G. A. Elliott's research grant. The first author acknowledges the support of an AARMS Postdoctoral Fellowship, while the third author acknowledges the support of a Juan de la Cierva Postdoctoral Fellowship.


\vspace{10pt}

\begin{bibdiv}
\begin{biblist}
\bib{b-p-t}{article}{
   author={Brown, N. P.},
   author={Perera, F.},
   author={Toms, A. S.},
   title={The Cuntz semigroup, the Elliott conjecture, and dimension
   functions on $C\sp *$-algebras},
   journal={J. Reine Angew. Math.},
   volume={621},
   date={2008},
   pages={191--211},
}

\bib{ciuperca}{article}{
   author={Ciuperca, A.},
   title={Some properties of the Cuntz semigroup and an isomorphism
   theorem for a certain class of non-simple $C\sp *$-algebras, {\rm Thesis,}},
   journal={University of Toronto, 2008},
   status={},
}

\bib{c-e}{article}{
   author={Ciuperca, Alin},
   author={Elliott, George A.},
   title={A remark on invariants for $C\sp *$-algebras of stable rank one},
   journal={Int. Math. Res. Not. IMRN},
   date={2008},
   number={5},
   pages={33},
}

\bib{c-e-i}{article}{
   author={Coward, K. T.},
   author={Elliott, G. A.},
   author={Ivanescu, C.},
   title={The Cuntz semigroup as an invariant for $C\sp *$-algebras},
   journal={J. Reine Angew. Math.},
   volume={623},
   date={2008},
   pages={161--193},
}



\bib{Cuntz}{article}{
   author={Cuntz, J.},
   title={Dimension functions on simple $C\sp*$-algebras},
   journal={Math. Ann.},
   volume={233},
   date={1978},
   number={2},
   pages={145--153},
}


\bib{Dixmier}{book}{
   author={Dixmier, Jacques},
   title={$C\sp*$-algebras},
   note={Translated from the French by Francis Jellett;
   North-Holland Mathematical Library, Vol. 15},
   publisher={North-Holland Publishing Co.},
   place={Amsterdam},
   date={1977},
}


\bib{elliott-AI}{article}{
   author={Elliott, G. A.},
   title={A classification of certain simple $C\sp *$-algebras},
   conference={
      title={Quantum and non-commutative analysis},
      address={Kyoto},
      date={1992},
   },
   book={
      series={Math. Phys. Stud., Vol. 16},
      publisher={Kluwer Acad. Publ.},
      place={Dordrecht},
   },
   date={1993},
   pages={373--385},
}

\bib{elliott-AT}{article}{
   author={Elliott, G. A.},
   title={On the classification of $C\sp *$-algebras of real rank zero},
   journal={J. Reine Angew. Math.},
   volume={443},
   date={1993},
   pages={179--219},
}


\bib{elliott-cancellation}{article}{
   author={Elliott, G. A.},
   title={Hilbert modules over a $C\sp *$-algebra of stable rank one},
   journal={C. R. Math. Acad. Sci. Soc. R. Can.},
   volume={29},
   date={2007},
   number={2},
   pages={48--51},
}

\bib{elliott-cla}{article}{
   author={Elliott, G. A.},
   title={Towards a theory of classification},
   journal={ Adv. Math.},
   volume={223},
   date={2010},
   number={1},
   pages={30--48},
}


\bib{e-ivan}{article}{
    author={Elliott, G. A.},
    author={Ivanescu, C.},
    title={The classification of separable simple C*-algebras which are inductive limits of continuous-trace C*-algebras with spectrum homeomorphic to the closed interval $[0,1]$},
    journal={J. Funct. Anal.},
    volume={254},
    date={2008},
    number={4},
    pages={879--903},
}

\bib{e-r-s}{article}{
author={Elliott, G. A.},
author={Robert, L.},
author={Santiago, L.},
title={The cone of lower semicontinuous traces on a $C\sp *$-algebra},
status={preprint, arXiv:0805.3122v2},
date={2008},
}


\bib{k-r}{article}{
   author={Kirchberg, E.},
   author={R{\o}rdam, M.},
   title={Infinite non-simple $C\sp *$-algebras: absorbing the Cuntz
   algebras $\scr O\sb \infty$},
   journal={Adv. Math.},
   volume={167},
   date={2002},
   number={2},
   pages={195--264},
}


\bib{Li}{article}{
   author={Li, L.},
   title={Classification of simple $C\sp *$-algebras: inductive limits of
   matrix algebras over trees},
   journal={Mem. Amer. Math. Soc.},
   volume={127},
   date={1997},
   number={605},
}



\bib{Loring}{book}{
   author={Loring, T. A.},
   title={Lifting solutions to perturbing problems in $C\sp *$-algebras},
   series={Fields Institute Monographs, Vol. 8},
   publisher={American Mathematical Society},
   place={Providence, RI},
   date={1997},
}



\bib{Loring1}{article}{
   author={Loring, T. A.},
   title={Stable relations. II. Corona semiprojectivity and dimension-drop
   $C\sp *$-algebras},
   journal={Pacific J. Math.},
   volume={172},
   date={1996},
   number={2},
   pages={461--475},
}

\bib{Pedersen}{book}{
   author={Pedersen, Gert K.},
   title={$C\sp{\ast} $-algebras and their automorphism groups},
   series={London Mathematical Society Monographs, Vol. 14},
   publisher={Academic Press Inc. [Harcourt Brace Jovanovich Publishers]},
   place={London},
   date={1979},
}


\bib{robert}{article}{
   author={Robert, L.},
   title={Classification of non-simple approximate interval
          $C\sp *$-algebras: the triangular case, {\rm Thesis,}},
   journal={University of Toronto, 2006},
   status={},
}

\bib{leonel}{article}{
   author={Robert, L.},
   title={The Cuntz semigroup of some spaces of dimension at most 2},
   journal={C. R. Math. Acad. Sci. Soc. R. Can.},
   status={to appear},

}

\bib{l-l}{article}{
   author={Robert, L.},
   author={Santiago, L.},
   title={On the classification of $C\sp*$-homomorphisms from $\mathrm{C}_0(0,1]$ to a $C\sp *$-algebra
of stable rank greater than 1},
   journal={ J. Funct. Anal.},
   status={to appear},
}


\bib{r-w}{article}{
   author={R{\o}rdam, M.},
   author={Winter, W.},
   title={The Jiang-Su algebra revisited},
   journal={ J. Reine Angew. Math.},
   status={to appear},
}

\bib{santiago}{article}{
   author={Santiago, L.},
   title={Classification of non-simple $C\sp *$-algebras: Inductive limits
   of splitting interval algebras, {\rm Thesis,}},
   journal={University of Toronto, 2008},
   status={},
}

\bib{stevens}{article}{
   author={Stevens, K. H.},
   title={The classification of certain non-simple approximate interval
   algebras},
   conference={
      title={Operator algebras and their applications, II (Waterloo, ON,
      1994/1995)},
   },
   book={
      series={Fields Inst. Commun., Vol. 20},
      publisher={Amer. Math. Soc.},
      place={Providence, RI},
   },
   date={1998},
   pages={105--148},
}

\bib{thomsen}{article}{
   author={Thomsen, K.},
   title={Inductive limits of interval algebras: unitary orbits of positive
   elements},
   journal={Math. Ann.},
   volume={293},
   date={1992},
   number={1},
   pages={47--63},
}


\bib{toms}{article}{
   author={Toms, A. S.},
   title={An infinite family of non-isomorphic $C\sp *$-algebras with
   identical $K$-theory},
   journal={Trans. Amer. Math. Soc.},
   volume={360},
   date={2008},
   number={10},
   pages={5343--5354},
}
\end{biblist}
\end{bibdiv}
\end{document}